\documentclass[11pt,onecolumn]{article}

\usepackage{a4wide}
\usepackage{amsmath,amsthm,epsfig,amssymb,amsbsy}
\usepackage{enumerate}
\usepackage{comment}
\usepackage{algorithm}
\usepackage{algorithmic}
\usepackage{amsfonts}       
\usepackage{dsfont}
\usepackage{enumitem}
\usepackage{amssymb}
\usepackage{multirow}
\usepackage{multicol}
\usepackage{booktabs} 

\usepackage{pgfplots}
\pgfplotsset{compat=1.17}
\usepackage{subfig}

\usepackage[
giveninits=true,
maxbibnames=9,
maxcitenames=2,
backend=bibtex,
bibstyle=numeric,
doi=false,isbn=false,url=false,
]{biblatex}
\renewbibmacro{in:}{
    \ifentrytype{article}
    {}
    {\printtext{\bibstring{in}\intitlepunct}}}
\newenvironment{keywords}{\begin{paragraph}{Keywords:}
}
{
\end{paragraph}
}
    
\newenvironment{MSCcodes}{\begin{paragraph}{AMS Subject Classification:}
}
{\end{paragraph}
}

\DeclareMathOperator*{\argmin}{arg\,min}

\DeclareMathOperator*{\argminimax}{arg\,minimax}

\DeclareMathOperator{\gph}{gph}

\DeclareMathOperator{\dom}{dom}

\DeclareMathOperator{\ran}{rge}

\DeclareMathOperator{\zer}{zer}
\DeclareMathOperator{\id}{Id}

\newcommand{\bR}{\mathbb{R}}

\newcommand{\bN}{\mathbb{N}}

\makeatletter
\newcommand{\aprox}[3][\@nil]{%
  \def\tmp{#1}%
   \ifx\tmp\@nnil
       \operatorname{prox}_{#3}^{#2}
    \else
         \operatorname{prox}_{#3}^{#1 \star #2}
    \fi}

\newcommand{\aenv}[3][\@nil]{%
  \def\tmp{#1}%
   \ifx\tmp\@nnil
       \operatorname{env}_{#3}^{#2}
    \else
         \operatorname{env}_{#3}^{#1 \star #2}
    \fi}

\newcommand{\bprox}[3][\@nil]{%
  \def\tmp{#1}%
   \ifx\tmp\@nnil
       \operatorname{bprox}_{#3}^{#2}
    \else
        \operatorname{bprox}_{#3}^{#1 #2}
    \fi}

\newcommand{\dist}[3][\@nil]{%
  \def\tmp{#1}%
   \ifx\tmp\@nnil
       d(#3, #2)
    \else
        d^{#1}(#3, #2)
    \fi}

\newcommand{\adist}[3][\@nil]{%
  \def\tmp{#1}%
   \ifx\tmp\@nnil
       d_p(#3, #2)
    \else
        d_p^{#1}(#3, #2)
    \fi}
    
\makeatother

\newcommand{\ares}[2]{\operatorname{ares}_{#2}^{#1}}
\newcommand{\bres}[2]{\operatorname{bres}_{#2}^{#1}}
\newcommand{\proj}[1]{\Pi_{#1}}

\newcommand{\aproj}[1]{\Pi_{#1}^{\phi}}

\usepackage{hyperref}
\hypersetup{
    colorlinks=true,
    linkcolor=blue,
    filecolor=magenta,
    citecolor =magenta,  
    urlcolor=magenta,
    pdftitle={Anisotropic Proximal Point Algorithm}
    }
\usepackage[capitalize]{cleveref}[0.19]
\usepackage{aliascnt}

\crefname{section}{section}{sections}
\crefname{subsection}{subsection}{subsections}
\Crefname{section}{Section}{Sections}
\Crefname{subsection}{Subsection}{Subsections}

\Crefname{figure}{Figure}{Figures}

\crefformat{equation}{\textup{#2(#1)#3}}
\crefrangeformat{equation}{\textup{#3(#1)#4--#5(#2)#6}}
\crefmultiformat{equation}{\textup{#2(#1)#3}}{ and \textup{#2(#1)#3}}
{, \textup{#2(#1)#3}}{, and \textup{#2(#1)#3}}
\crefrangemultiformat{equation}{\textup{#3(#1)#4--#5(#2)#6}}%
{ and \textup{#3(#1)#4--#5(#2)#6}}{, \textup{#3(#1)#4--#5(#2)#6}}{, and \textup{#3(#1)#4--#5(#2)#6}}

\Crefformat{equation}{#2Equation~\textup{(#1)}#3}
\Crefrangeformat{equation}{Equations~\textup{#3(#1)#4--#5(#2)#6}}
\Crefmultiformat{equation}{Equations~\textup{#2(#1)#3}}{ and \textup{#2(#1)#3}}
{, \textup{#2(#1)#3}}{, and \textup{#2(#1)#3}}
\Crefrangemultiformat{equation}{Equations~\textup{#3(#1)#4--#5(#2)#6}}%
{ and \textup{#3(#1)#4--#5(#2)#6}}{, \textup{#3(#1)#4--#5(#2)#6}}{, and \textup{#3(#1)#4--#5(#2)#6}}

\newtheorem{theorem}{Theorem}[section]

\newlist{thmenum}{enumerate}{1} 
\setlist[thmenum]{label=(\roman*), ref=\thetheorem(\roman*), font=\rm} 
\crefalias{thmenumi}{theorem}

\newaliascnt{corollary}{theorem}
\newtheorem{corollary}[corollary]{Corollary}
\aliascntresetthe{corollary}

\newaliascnt{lemma}{theorem}
\newtheorem{lemma}[lemma]{Lemma}
\aliascntresetthe{lemma}

\newlist{lemenum}{enumerate}{1}
\setlist[lemenum]{label=(\roman*), ref=\thelemma(\roman*), font=\rm} 
\crefalias{lemenumi}{lemma}

\newaliascnt{proposition}{theorem}
\newtheorem{proposition}[proposition]{Proposition}
\aliascntresetthe{proposition}

\newlist{propenum}{enumerate}{1} 
\setlist[propenum]{label=(\roman*), ref=\theproposition(\roman*), font=\rm}
\crefalias{propenumi}{proposition}

\newaliascnt{definition}{theorem}
\newtheorem{definition}[definition]{Definition}
\aliascntresetthe{definition}

\newlist{defenum}{enumerate}{1}
\setlist[defenum]{label=(\roman*), ref=\thedefinition(\roman*), font=\rm}
\crefalias{defenumi}{definition}

\newlist{corenum}{enumerate}{1} 
\setlist[corenum]{label=(\roman*), ref=\thecorollary(\roman*), font=\rm} 
\crefalias{corenumi}{corollary}

\theoremstyle{remark}
\newaliascnt{remark}{theorem}
\newtheorem{remark}[remark]{Remark}
\aliascntresetthe{remark}

\newaliascnt{example}{theorem}
\newtheorem{example}[example]{Example}
\aliascntresetthe{example}

\newlist{assumenum}{enumerate}{1} 
\setlist[assumenum]{leftmargin=2.1cm,label=(A\arabic*),font=\bfseries}
\creflabelformat{assumenumi}{#2#1#3}
\crefname{assumenumi}{assumption}{assumptions}
\Crefname{assumenumi}{Assumption}{Assumptions}

\numberwithin{equation}{section}

\usepackage{xcolor}

\usepackage{xparse}
	\newif\ifshowcomments\showcommentstrue
	\makeatletter
		\newcommand{\disablecolorlinks}{\def\HyColor@UseColor##1{}}
		\newcommand{\defineInlineComment}[2][]{
			\expandafter\gdef\csname #2\endcsname##1{\ifshowcomments{\color{Blue}\footnotesize{\sc\hl[#1]{#2}: }{\it ##1}}\fi}
		}%
	\makeatother

	\newif\ifshowold\showoldtrue
	\newif\ifshownew\shownewtrue
	\newcounter{saveTheorem}
	\newcounter{saveEquation}
	\newcounter{saveSection}
	\newcounter{saveSubsection}
	\colorlet{newcolor}{orange!70!red}
	\colorlet{oldcolor}{black!30}

	\newcommand{\old}[1]{{%
		\disablecolorlinks
		\ifshowold
			{%
				\setcounter{saveEquation}{\value{equation}}%
				\setcounter{saveTheorem}{\value{theorem}}%
				\setcounter{saveSection}{\value{section}}%
				\setcounter{saveSubsection}{\value{subsection}}%
				\setcounter{saveSubsection}{\value{subsection}}%
				\renewcommand{\thetheorem}{\thesaveSection.\arabic{theorem}\textsuperscript{*}}%
				\renewcommand{\theequation}{\thesaveSection.\arabic{equation}\textsuperscript{*}}%
				\renewcommand{\thesection}{\Alph{section}\textsuperscript{*}}
				\renewcommand{\thesubsection}{\thesection.\Alph{subsection}\textsuperscript{*}}
				{\color{oldcolor}{}#1}%
				\setcounter{equation}{\value{saveEquation}}%
				\setcounter{theorem}{\value{saveTheorem}}%
				\setcounter{section}{\value{saveSection}}%
				\setcounter{subsection}{\value{saveSubsection}}%
			}%
		\fi
	}}
	\newcommand{\new}[1]{{%
		\disablecolorlinks
		\ifshownew
			{\color{newcolor}{}#1}%
		\fi
	}}
	\DeclareExpandableDocumentCommand{\change}{O{}m}{%
		{%
			\old{#1}%
		}\new{#2}%
	}

\title{Anisotropic Proximal Point Algorithm}
\author{Emanuel Laude\thanks{KU Leuven,
		Department of Electrical Engineering (ESAT-STADIUS),
		Kasteelpark Arenberg 10, 3001 Leuven, Belgium~
		{\tt%
			\href{mailto:emanuel.laude@esat.kuleuven.be}{\{emanuel.laude,}%
			\href{mailto:panos.patrinos@esat.kuleuven.be}{panos.patrinos\}}%
			\href{mailto:emanuel.laude@esat.kuleuven.be,panos.patrinos@esat.kuleuven.be}{@esat.kuleuven.be}%
		}
	} \and Panagiotis Patrinos\footnotemark[1]}

\begin{document}

\maketitle
\begin{abstract}
In this paper we study a nonlinear dual space preconditioning approach for the relaxed \emph{Proximal Point Algorithm} (PPA) with application to monotone and relatively cohypomonotone inclusions, called anisotropic PPA. The algorithm is an instance of Luque's nonlinear PPA wherein the nonlinear preconditioner is chosen as the gradient of a Legendre convex function. Since the preconditioned operator is nonmonotone in general, convergence cannot be shown using standard arguments, unless the preconditioner exhibits isotropy (preserves directions) as in existing literature. To address the broader applicability we show convergence along subsequences invoking a Bregman version of Fej\'er-monotonicity in the dual space.
Via a nonlinear generalization of Moreau's decomposition for operators, we provide a dual interpretation of the algorithm in terms of a forward iteration applied to a $D$-firmly nonexpansive mapping which involves the Bregman resolvent.
For a suitable preconditioner, convergence rates of arbitrary order are derived under a mild H\"older growth condition.
Finally, we discuss an anisotropic generalization of the proximal augmented Lagrangian method obtained via the proposed scheme. This aligns with Rockafellar's generalized and sharp Lagrangian functions.
\end{abstract}


\section{Introduction}
\subsection{Motivation and prior work}
This paper introduces a nonlinear dual space preconditioning scheme for the relaxed PPA for solving monotone and relatively cohypomonotone \cite{otero2007proximal} inclusion problems. The framework covers, as special cases, convex nonsmooth optimization, variational inequalities, and convex-concave saddle-point problems. In direct analogy with linear preconditioning for linear systems, the proposed nonlinear preconditioning approach aims to yield better algorithms for nonlinear inclusion problems.

Our proposed algorithm left-preconditions the operator with the gradient of a Legendre convex function, representing an instance of Luque's nonlinear PPA \cite{luque1984nonlinear,luque1987nonlinear} wherein the preconditioner itself is another maximal monotone mapping. However, to establish convergence using standard arguments, Luque's framework assumes isotropy in the preconditioner, which implies the preservation of the update direction in relation to the classical PPA update. This assumption can restrict the general applicability of the method. To address this limitation, we show convergence in the dual space by harnessing the gradient structure of the preconditioner.

A widely adopted approach for nonlinear extensions of the PPA is the Bregman PPA \cite{Teboulle92,Eckstein93}. Similar to the Bregman PPA formulated using the well-established Bregman or $D$-resolvent \cite{Teboulle92,bauschke2003bregman,bauschke2008general,solodov2000inexact,Eckstein93}, the anisotropic PPA can be expressed in terms of the anisotropic resolvent. In this work we reveal a relationship between the anisotropic resolvent and the Bregman resolvent, establishing a duality between them which can be be expressed within the more general framework of convolutions of maximal monotone operators \cite{luque1986convolutions,luque1984nonlinear} more commonly referred to as the parallel sum \cite{passty1986parallel}. Specifically, when the operator represents a subdifferential, this relationship manifests as a Bregman variant of Moreau's decomposition \cite{combettes2013moreau}. This allows one to interpret anisotropic PPA as a nonlinearly preconditioned forward iteration applied to a $D$-firmly nonexpansive mapping \cite{bauschke2003bregman}.  In light of recent work \cite{maddison2021dual,wang2021bregman,laude2021conjugate,laude2022anisotropic},  $D$-firm nonexpansiveness of a subdifferential operator is a characterization of anisotropic smoothness \cite{laude2021lower,laude2021conjugate,laude2022anisotropic} and the algorithm in \cite{maddison2021dual} can be interpreted in terms of a majorize-minimize scheme called anisotropic proximal gradient \cite{laude2022anisotropic}. The step-size in the forward iteration corresponds to a relaxation step in the PPA. Specifically, a non-unit relaxation parameter allows one to handle relatively cohypomonotone mappings \cite{otero2007proximal}, a generalization of cohypomonotoncity \cite{combettes2004proximal,evens2025convergence} which in turn generalizes classical monotonicity.

Bregman PPA and our anisotropic PPA differ by right versus left preconditioning: Bregman PPA applies the preconditioner on the right of the operator, while our method applies it on the left.

Beyond its theoretical scope, our framework specializes to a novel anisotropic proximal augmented Lagrangian method that permits sharper penalty functions and more flexible, trust-region-like primal regularization. This generalizes the well-known correspondence between the proximal point method and the proximal augmented Lagrangian method \cite{rockafellar1976augmented}.
As evidenced in our experiments, this often leads to faster convergence in practice.

Unlike the Bregman augmented Lagrangian method proposed in \cite{Eckstein93}, which corresponds to a mirror ascent step on a specific Bregman augmented Lagrangian function, our approach naturally aligns with Rockafellar’s generalized and sharp Lagrangian functions \cite{RoWe98}. Concretely, the dual update in our method takes the form of a nonlinearly preconditioned gradient ascent step, rather than a standard mirror ascent step. To the best of our knowledge, this perspective and algorithmic formulation of a proximal augmented Lagrangian method using anisotropic dual preconditioning has not yet been explored in the literature.

\subsection{Contributions}
The main objective of this manuscript is to study the convergence of the anisotropic proximal point algorithm with under-relaxation for monotone and cohypomonotone inclusions beyond the isotropic case and unit relaxation parameter as in \cite{luque1984nonlinear}. Since existing techniques do not apply in the anisotropic case, we prove global convergence along subsequences using an unusual technique that is based on a Bregman version of Fej\'er-monotonicity in the dual space.

We analyze the local convergence rate of our method under a mild H\"older growth condition restricting the preconditioner to the inverse gradient of the $\ell^p$-norm raised to the power of $p$ larger than $2$. If the growth of the preconditioner dominates the inverse growth of the operator we show that the algorithm converges locally with arbitrary order.

Finally, we discuss an anisotropic generalization of the proximal augmented Lagrangian method obtained via the proposed scheme. This aligns with Rockafellar's generalized and sharp Lagrangians \cite{RoWe98}.

\subsection{Standing assumptions and notation}
Let $X$ be a Euclidean space. Although primal and dual space in this case are identical, to improve clarity we will stick to the convention that dual objects are elements of $X^*=X$. Hence, we interpret the Euclidean inner product in terms of the dual pairing $\langle \cdot,\cdot \rangle : X \times X^* \to \bR$ and denote by $\|x\|_2=\sqrt{\langle x,x \rangle}$ the Euclidean norm.
$T$ is a set-valued mapping from $X$ to $X^*$, written $T: X \rightrightarrows X^*$ if $T$ is a mapping from $X$ to the power set of $X^*$, i.e., for all $x \in X$ we have $T(x) \subseteq X^*$. The graph of $T$ is given as $\gph T = \{ (x,x^*) \in X \times X^* : x^* \in T(x)\}$. Such a mapping is monotone if for all $(x,x^*),(y,y^*) \in \gph T$ we have $\langle x-y, x^* - y^*\rangle \geq 0$, and maximal monotone if its graph is not strictly contained in the graph of another monotone mapping. The domain of $T$ is $\dom T= \{x \in X : T(x) \neq \emptyset \}$ and the range of $T$ is $\ran T= T(X) = \bigcup_{x \in X} T(x)$. There always exists an inverse of $T$ which is the set-valued mapping $T^{-1}: X^* \rightrightarrows X$, defined via $T^{-1}(x^*):=\{x \in X : x^*\in T(x)\}$. We have the relation $\dom T = \ran T^{-1}$. We define the set of solutions or zeros of $T$ as $\zer T := T^{-1}(0)= \{ x \in X : 0 \in T(x)\}$. Let $\phi: X \to \bR$ be strictly convex, smooth and super-coercive, i.e., $\phi(x)/\|x\|\to \infty$ whenever $\|x\|\to\infty$. Then the same is true for the convex conjugate $\phi^*: X^* \to \bR$, and its gradient is the bijection $\nabla \phi : X \to X^*$ with inverse $\nabla \phi^*$. We refer to $\phi$ as the prox-function in the remainder of the manuscript. Throughout we assume that $\nabla \phi(0) = 0$,
and hence $\nabla \phi^*(0) = 0$. Note that this is not restrictive since we can always redefine $\phi$ as $\phi(x) - \langle x, \nabla \phi(0) \rangle$.
An important example is the separable function  $\phi(x)=\frac{1}{p}\|x\|_p^p$ with $p> 2$. Then $\phi^*(v)=\frac{1}{q}\|v\|_q^q$ with $\frac{1}{p} + \frac{1}{q} = 1$. Other interesting examples are $\phi(x) = \sum_{i=1}^n \cosh(x_i)$ or $\phi(x)=\sum_{i=1}^n (\exp(|x_i|) - |x_i| - 1)$.
In addition we introduce the epi-scaling $\tau \star \phi$ of $\phi$: $(\tau \star \phi)(x) =
        \tau \phi(\tau^{-1} x)$ if $\tau > 0$ and $\delta_{\{0\}}(x)$ otherwise. Its convex conjugate amounts to $(\tau \star \phi)^* = \tau \phi^*.$ Denote by $(f \mathbin{\square} \phi)(y):=\inf_{x\in X} f(x) + \phi(y-x)$ the infimal convolution of $f$ and $\phi$. We say that $f \mathbin{\square} \phi$ is exact at $y$ for $x$ if $(f \mathbin{\square} \phi)(y)=f(x)+\phi(y-x)$. Let $C \subset X$ be a closed convex nonempty set. Then we denote by $\proj{C}(y):=\argmin_{x \in C}\|x-y\|_2$ the Euclidean projection of $x$ onto $C$ and $\dist{C}{x}:=\inf_{y \in \zer T}\|x-y\|_2$ is the Euclidean distance of $x$ to $C$. By $N_C(y):=\{v \in X^* : \langle x-y,v \rangle \leq 0,\;  \forall x \in C\}$ we denote the normal cone of $C$ at $y \in C$ with the convention $N_C(y)=\emptyset$ if $y \notin C$.

\subsection{Paper organization}
The remainder of the paper is organized as follows:

\Cref{sec:ppa} describes the relaxed anisotropic PPA that is studied in this paper.
\Cref{sec:duality} examines a duality between the anisotropic resolvent and the Bregman resolvent and reveals an equivalence between the anisotropic PPA and a nonlinearly preconditioned forward iteration applied to a $D$-firmly nonexpansive mapping. Furthermore, it establishes connections to the Bregman--Yosida regularization and the framework of operator convolutions also known as \emph{parallel sums}.
\Cref{sec:convergence} studies the convergence of the algorithm where our main contributions are in \cref{sec:convergence_anisotropic}.
\Cref{sec:alm} provides an example application with numerical simulations in terms of an anisotropic proximal \emph{Augmented Lagrangian Method} (ALM) and draws connections to Rockafellar's generalized and sharp Lagrangians.
\Cref{sec:conclusion} concludes the paper.

\section{Anisotropic Proximal Point Algorithm with relaxation} \label{sec:ppa}
In this paper we investigate a nonlinear proximal point algorithm for solving monotone inclusion problems that are defined as follows: given a maximal monotone mapping $T : X \rightrightarrows X^*$ we seek to compute a point $x^\star \in X$, such that $0 \in T(x^\star)$.
We consider the following proximal point-type algorithm with under-relaxation parameter $\lambda \in (0, 1]$ which, at each iteration $k$, updates the variable $x^k$ by computing a triplet $(x^{k+1}, z^k,v^k) \in X \times X \times X^*$ that satisfies
\begin{subequations} \label{eq:anisotropic_ppa}
\begin{align}
v^k &\in T(z^{k}) \\
z^{k} &= x^k - \nabla \phi^*(v^k) \\
x^{k+1} &= x^k + \lambda (z^k - x^k).
\end{align}
\end{subequations}
It is easy to see that fixed-points $x^\star$ of the algorithm \cref{eq:anisotropic_ppa} satisfy $0 \in (\nabla \phi^* \circ T)(x^\star)$.
Since $\nabla \phi^*(v)=0\Leftrightarrow v= 0$,
\begin{equation}
0 \in (\nabla \phi^* \circ T)(x^\star) \Leftrightarrow 0 \in T(x^\star),
\end{equation}
and hence fixed-points of the algorithm are zeros of $T$.

Next, we define the \emph{anisotropic} resolvent of $T$:
\begin{definition} \label{def:anisotropic_resolvent}
    Let $T: X \rightrightarrows X^*$ be a set-valued mapping. Then the mapping $\ares{\phi}{T} : X \rightrightarrows X$ defined by
    $$
    \ares{\phi}{T} := (\id + \nabla \phi^* \circ T)^{-1},
    $$
    is called the anisotropic resolvent of $T$.
\end{definition}
Using the above definition the algorithm can be rewritten more compactly as
\begin{align}
x^{k+1} \in (1- \lambda) x^k + \lambda \ares{\phi}{T}(x^k) \label{eq:update}.
\end{align}

Furthermore, this shows that the proposed algorithm is identical to the classical relaxed \emph{Proximal Point Algorithm} (PPA) applied to the mapping $\nabla \phi^* \circ T$. In spite of $\nabla \phi^*$ and $T$ being monotone, the composition $\nabla \phi^* \circ T$ is potentially nonmonotone as shown by the following counter-example:
\begin{example} \label{ex:counter_example_monotone}
Choose the monotone linear operator $T:\bR^2 \to \bR^2$ defined by the matrix
\begin{align}
    T = \begin{bmatrix} 0 & 1 \\ -1 & 0\end{bmatrix},
\end{align}
and the convex and smooth prox-function $\phi:\bR^2 \to \bR$ with $\phi=\frac{1}{p}\|\cdot\|_2^p$ for $p>2$. Hence, $\phi^*(v)=\frac{1}{q}\|v\|_2^q$ with $q=\frac{p}{p-1}$. Then $\nabla \phi^*(v)=\|v\|_2^{-\frac{p-2}{p-1}} v$ for $v\neq 0$ and $\nabla \phi^*(0)=0$. For any $(x_1,x_2) \in \bR^2$ we have
$$
\nabla \phi^*(T(x_1,x_2)) = \|(x_1,x_2)\|_2^{-\frac{p-2}{p-1}} (x_2, -x_1)
$$
and thus for any $(x_1,x_2),(y_1,y_2) \in \bR^2$ it holds
\begin{align*}
&\langle \nabla \phi^*(T(x_1,x_2)) - \nabla \phi^*(T(y_1,y_2)), (x_1,x_2) -(y_1,y_2) \rangle \\
&\qquad=\Big(\|(x_1,x_2)\|_2^{-\frac{p-2}{p-1}}- \|(y_1,y_2)\|_2^{-\frac{p-2}{p-1}}\Big) (x_1 y_2 - x_2 y_1).
\end{align*}
Choose $(x_1,x_2)=(2,0)$ and $(y_1,y_2)=(1, 1)$. Then we obtain
\begin{align*}
&\langle \nabla \phi^*(T(x_1,x_2)) - \nabla \phi^*(T(y_1,y_2)), (x_1,x_2) -(y_1,y_2) \rangle =(2^{-\frac{p-2}{p-1}}- \sqrt{2}^{-\frac{p-2}{p-1}}) 2 < 0,
\end{align*}
since $t \mapsto t^{-\frac{p-2}{p-1}}$ is monotonically decreasing.

\end{example}
In light of the above example, existing convergence results for classical PPA do not apply in our case.

For $T =\partial f$ being the subdifferential of a convex, proper lsc function $f:X \to \bR \cup\{\infty\}$, anisotropic PPA with $\lambda=1$ corresponds to the update
$
x^{k+1} = \argmin_{x \in X} f(x) + \phi(x^k - x)
$. This method has been studied under the more general assumption of (anisotropic) prox-regularity in \cite{laude-wu-cremers-aistats-19,laude2021lower}.
For $\phi=\frac{1}{p}\| \cdot \|_2^p$, $p>2$, this is identical to the higher-order proximal point algorithm studied by \cite{nesterov2021inexact}. More generally, if $\phi=\frac{1}{p}\| \cdot \|_p^p$, $p>2$ the function minimization case is recently explored in \cite{oikonomidis2025global}.

We now turn to the relationship between the well-known Bregman proximal point algorithm \cite{Eckstein93,solodov2000inexact} and the anisotropic PPA (with $\lambda=1$), highlighting how both can be interpreted as nonlinearly preconditioned variants of the classical PPA, applied in primal and dual spaces, respectively. In the Bregman PPA, each iteration $k$ updates $x^{k}$ via the recursion $x^{k+1}\in(\nabla \phi + T)^{-1} \circ \nabla \phi(x^k)$.

Introducing dual variables $z^k:=\nabla\phi(x^k)$, this recursion can be written as the resolvent of the right-preconditioned operator $T\circ\nabla\phi^*$:
$
z^{k+1} \in (\id + T \circ\nabla\phi^*)^{-1}(z^k).
$
By contrast, our anisotropic PPA applies the preconditioner on the left and updates the primal variables by the resolvent of $\nabla\phi^*\circ T$.
Thus, Bregman PPA computes the resolvent of the right-preconditioned operator $T \circ \nabla\phi^*$ in the dual space, whereas our method computes the resolvent of the left-preconditioned operator $\nabla\phi^*\circ T$ in the primal space.

In the special case where $\phi(x)=\frac{1}{2}\langle x, Ax \rangle$, for a symmetric positive definite matrix $A$, the gradient mappings are linear: $\nabla \phi(x)= A x$ and $\nabla \phi^*(v) = A^{-1} v$. As a result, both algorithms coincide.

In the next section, we explore a different facet of the connection between the Bregman and anisotropic resolvents: a duality expressed through a generalized Moreau decomposition.

\section{Duality and the \texorpdfstring{$D$}{D}-resolvent} \label{sec:duality}
In this section we shall discuss a duality relation between the anisotropic resolvent (\cref{def:anisotropic_resolvent}) and the Bregman resolvent or $D$-resolvent \cite[Definition 3.7]{bauschke2003bregman} and revisit the notion of $D$-firm nonexpansiveness \cite[Definition 3.4]{bauschke2003bregman}.
\begin{definition}
    Let $S: X^* \rightrightarrows X$ be a set-valued mapping. Then the mapping $\bres{\phi^*}{S} : X^* \rightrightarrows X^*$ defined by
    $$
    \bres{\phi^*}{S} := (\nabla \phi^* + S)^{-1} \circ \nabla \phi^*,
    $$
    is called the Bregman resolvent or $D$-resolvent of $S$.
\end{definition}

In light of \cite[Example 3.1(iv)]{bauschke2008general}, $F=\nabla \phi^*$ satisfies the standing assumption in \cite{bauschke2008general} from where we adapt the following definitions and results.
We recall the definition of $\nabla \phi^*$-firm nonexpansiveness \cite[Definition 3.2]{bauschke2008general}:
\begin{definition}
Let $C\subseteq X^*$ and $A:C \to X^*$. Then $A$ is $\nabla \phi^*$-firmly nonexpansive if for all $x,y \in C$ we have
$$
\langle \nabla \phi^*(A(x)) - \nabla \phi^*(A(y)),A(x) - A(y)\rangle \leq \langle \nabla \phi^*(x) - \nabla \phi^*(y), A(x) - A(y)\rangle.
$$
\end{definition}
In the Euclidean case, i.e., $\nabla \phi^* = \id$ this reduces to the classical notion of firm nonexpansiveness
\begin{align}
\|A(x) - A(y)\|^2 \leq \langle x - y, A(x) - A(y)\rangle,
\end{align}
which can be rewritten by completing the square
\begin{align}\label{eq:euclidean_nonexpansiveness}
\|A(x) - A(y)\|^2 + \|(\id -A)(x) - (\id - A)(y)\|^2\leq \|x-y\|^2.
\end{align}
In light of \cite[Remark 3.3]{bauschke2008general}, $\nabla \phi^*$-firmly nonexpansive mappings are at most single-valued.
We have the following specialization of \cite[Proposition 4.2(iii)]{bauschke2008general}:
\begin{proposition}\label{thm:firm_nonexpansive}
Let $S: X^* \rightrightarrows X$ be a set-valued monotone mapping. Then $\bres{\phi^*}{S} := (\nabla \phi^* + S)^{-1} \circ \nabla \phi^*$ is at most single-valued and $\nabla \phi^*$-firmly nonexpansive.
\end{proposition}
Next we state a specialization of \cite[Proposition 5.1]{bauschke2008general} showing that any $\nabla \phi^*$-firmly nonexpansive mapping $A$ can be written in terms of the Bregman resolvent of a monotone mapping:
\begin{proposition} \label{thm:inverse_resolvent}
Let $C\subseteq X^*$ and $A:C \to X^*$, and set $F_A =\nabla \phi^* \circ A^{-1} - \nabla \phi^*$. Then
\begin{propenum}
    \item \label{thm:inverse_resolvent:bregman} The Bregman resolvent of $F_A$ is A: $\bres{\phi^*}{F_A}=A$.
    \item \label{thm:inverse_resolvent:bregman2} If $A$ is $\nabla \phi^*$-firmly nonexpansive, then $F_A$ is monotone.
    \item \label{thm:inverse_resolvent:bregman3} If $A$ is $\nabla \phi^*$-firmly nonexpansive, then $C = X^* \Leftrightarrow F_A$ is maximal monotone.
\end{propenum}
\end{proposition}

Next we prove the following relation between the anisotropic resolvent of $T$ and the Bregman resolvent of $T^{-1}$, which generalizes Moreau's decomposition for the Bregman proximal mapping \cite{combettes2013moreau} to operators:
\begin{proposition}[Moreau's decomposition] \label{thm:moreau_decomposition}
Let $T: X \rightrightarrows X^*$ be a set-valued mapping. Then the following identity holds:
\begin{align}
\ares{\phi}{T}  = \id - \nabla \phi^* \circ \bres{\phi^*}{T^{-1}} \circ \nabla \phi.
\end{align}
If, furthermore, $T$ is monotone, $\ares{\phi}{T}$ is single-valued.
\end{proposition}
\begin{proof}
We have that:
\begin{align*}
x \in \ares{\phi}{T}(y) &\Leftrightarrow x \in (\id + \nabla \phi^* \circ T)^{-1}(y) \\
&\Leftrightarrow y \in (\id + \nabla \phi^* \circ T)(x) \\
&\Leftrightarrow \nabla \phi(y - x) \in T(x) \\
&\Leftrightarrow x \in T^{-1}(\nabla \phi(y - x)) \\
&\Leftrightarrow \nabla \phi^*(\nabla \phi(y)) - \nabla \phi^*(\nabla \phi(y-x)) \in T^{-1}(\nabla \phi(y - x)) \\
&\Leftrightarrow \nabla \phi^*(\nabla \phi(y)) \in (\nabla \phi^* + T^{-1})(\nabla \phi(y - x)) \\
&\Leftrightarrow \nabla \phi(y - x) \in ((\nabla \phi^* + T^{-1})^{-1} \circ \nabla \phi^*)(\nabla \phi(y)) \\
&\Leftrightarrow x \in y-\nabla \phi^*\big(((\nabla \phi^* + T^{-1})^{-1} \circ \nabla \phi^*)(\nabla \phi(y))\big) \\
&\Leftrightarrow x \in y-\nabla \phi^*(\bres{\phi^*}{T^{-1}}(\nabla \phi(y)).
\end{align*}
If, furthermore, $T$ is monotone, $T^{-1}$ is monotone as well and thus, by \cref{thm:firm_nonexpansive}, $\ares{\phi}{T}$ is single-valued.
\end{proof}
Alternatively the result follows from \cite[Proposition 2.8]{luque1986convolutions} with $S=\nabla \phi$ being single-valued.
In light of \cref{thm:moreau_decomposition}, anisotropic proximal point \cref{eq:update} corresponds to the forward iteration
\begin{align} \label{eq:fix_point_iter_nonexpansive}
x^{k+1} = (1- \lambda) x^k + \lambda \ares{\phi}{T}(x^k)= x^k - \lambda(\nabla \phi^* \circ \bres{\phi^*}{T^{-1}}\circ \nabla \phi)(x^k).
\end{align}

Conversely, by the above results, if $A$ is $\nabla \phi^*$-firmly nonexpansive, the forward iteration $x^{k+1}=x^k - \lambda(\nabla \phi^* \circ A \circ \nabla \phi)(x^k)$ is anisotropic PPA \cref{eq:update} applied to the monotone operator $T=F_A^{-1}$ with $F_A = \nabla \phi^* \circ A^{-1} - \nabla \phi^*$:
\begin{align} \label{eq:forward_iteration}
x^{k+1}&=x^k - \lambda (\nabla \phi^* \circ A \circ \nabla \phi)(x^k) \\
&=x^k - \lambda(\nabla \phi^* \circ \bres{\phi^*}{F_A} \circ \nabla \phi)(x^k) \notag \\
&= (1- \lambda) x^k + \lambda \ares{\phi}{F_A^{-1}}(x^k) \notag,
\end{align}
where the second equality follows from \cref{thm:inverse_resolvent:bregman} and the last equality from \cref{thm:moreau_decomposition}.
This equivalence reveals that the convergence of anisotropic PPA \cref{eq:update} entails the convergence of the forward iteration \cref{eq:forward_iteration} for the $\nabla \phi^*$-firmly nonexpansive mapping $A$ (and vice versa).
If $A = \nabla f \circ \nabla \phi^*$ for $f$ being anisotropically smooth \cite{laude2021conjugate,laude2022anisotropic}, the algorithm corresponds to the recent anisotropic proximal gradient descent method for $g=0$  with step-size $\lambda$, \cite{maddison2021dual,laude2022anisotropic}.
Next we define the Bregman--Yosida regularization of $T$ \cite[Definition 2]{otero2007proximal} with parameter $\rho\geq 0$ which, for $\rho>0$, amounts to the operator convolution of $T$ and $\nabla(\rho \star \phi)$ in the sense of \cite{luque1986convolutions,luque1984nonlinear}, also known as the parallel sum \cite{passty1986parallel} of the two operators:
\begin{definition}[Bregman--Yosida regularization]
  Let $T: X \rightrightarrows X^*$ be a set-valued mapping. Choose $\rho \geq 0$. Then we call the mapping $T_\rho : X \rightrightarrows X^*$ defined by
  \begin{align}
  T_\rho:= (\rho \nabla \phi^* + T^{-1})^{-1}
  \end{align}
  the Bregman--Yosida regularization of $T$ with parameter $\rho$.
\end{definition}
Next we will provide equivalent characterizations for the monotonicity of $T_\rho$.
Following \cite{otero2007proximal} we introduce the notion of relative cohypomonotonicity generalizing \cite{combettes2004proximal,evens2025convergence}:
\begin{definition}[relative (co)-hypomonotonicity]
    Let $S: X^* \rightrightarrows X$ be a set-valued mapping. Then $S$ is said to be (maximal) hypomonotone relative to $\nabla \phi^*$ with parameter $\rho \geq 0$ if
    $$
    \langle u - v, x-y\rangle \geq -\rho\langle\nabla \phi^*(x) - \nabla \phi^*(y), x- y \rangle,
    $$
    for all $(x,u),(y,v) \in \gph S$. We say that the set-valued mapping $T: X \rightrightarrows X^*$ is (maximal) cohypomonotone relative to $\nabla \phi$ with parameter $\rho \geq 0$ if $T^{-1}$ is (maximal) hypomonotone relative to $\nabla \phi^*$ with parameter $\rho$.
    
\end{definition}
Clearly, every monotone mapping $S$ is relatively hypomonotone for any $\rho \geq 0$.
Next we state \cite[Lemma 1]{otero2007proximal}:
\begin{proposition}
\label{thm:yosida}
Let $T: X \rightrightarrows X^*$ be a set-valued mapping. Then we have that
\begin{propenum}
\item \label{thm:moreau_decomposition:fixpoint} $
0 \in T(x^\star)\Leftrightarrow 0 \in T_\rho(x^\star)
$, for any $\rho \geq 0$.
\item \label{thm:moreau_decomposition:monotone} $T_\rho$ is (maximal) monotone if and only if $T$ is (maximal) cohypomonotone relative to $\nabla \phi$ with parameter $\rho$.

\end{propenum}
\end{proposition}
Since the regularized operator
$T_\rho$ shares the same zeros as $T$ and becomes monotone when
$T$ is relatively cohypomonotone (as established above), it is naturally suited as a surrogate for solving such inclusions: In \cite{otero2007proximal} the Bregman PPA is applied to the monotone mapping $T_\rho$ to find a zero of a relatively cohypomonotone mapping $T$. However, in light of \cite[Lemma 2(iv)]{otero2007proximal}, this approach requires the solution of a complicated system of inclusions at each iteration. In contrast, the following result reveals that applying the anisotropic PPA to $T_\rho$ leads only to an under-relaxation step, yielding a simpler and more natural generalization of the Euclidean case \cite{evens2025convergence}:
\begin{proposition} \label{thm:absorb_relaxation_stepsize}
    Let $T: X \rightrightarrows X^*$ be a set-valued mapping and let $\rho \geq 0$ and $\tau >0$. Then the following identity holds:
    \begin{equation}
    (\id + \tau \nabla \phi^* \circ T_\rho)^{-1} = (1 - \lambda) \id + \lambda(\id + \gamma\nabla \phi^* \circ T)^{-1},
    \end{equation}
    for $\gamma = \tau + \rho$ and $\lambda = \tfrac{\tau}{\gamma} \in (0, 1]$.
\end{proposition}
\begin{proof}
Note that $T_\gamma = \bres{\gamma\phi^*}{T^{-1}} \circ \nabla (\gamma \star\phi)$.
Invoking \cref{thm:moreau_decomposition} and reminding ourselves that $(\gamma \star \phi)^* = \gamma \phi^*$ we have:
\begin{align} \label{eq:identity_tgamma}
\gamma \nabla \phi^*(T_\gamma(x))&= (x - (\id + \gamma \nabla \phi^* \circ T)^{-1}(x)).
\end{align}
Define $T_\rho^{-1}:= (T_\rho)^{-1} = \rho \nabla \phi^* + T^{-1}$.
Then we have for any $x\in X$ the identity
\begin{align*}
(1 - \lambda) x + \lambda(\id + \gamma \nabla \phi^* \circ T)^{-1}(x) &= x - \lambda(x - (\id + \gamma \nabla \phi^* \circ T)^{-1}(x)) \\
&= x - \tau \nabla \phi^*(T_\gamma(x)) \\
&= x - \tau \nabla \phi^*\big((\tau \nabla \phi^* + (\rho \nabla \phi^* + T^{-1}))^{-1}(x)\big) \\
&= x - \tau \nabla \phi^*\big((\tau \nabla \phi^* + T_\rho^{-1})^{-1}(x)\big) \\
&= x - \tau \nabla \phi^*\big(\bres{\tau \phi^*}{T_\rho^{-1}}(\nabla (\tau \star \phi)(x))\big) \\
&=(\id + \tau \nabla \phi^* \circ T_\rho)^{-1}(x),
\end{align*}
where the second identity follows from \cref{eq:identity_tgamma} since $\lambda \gamma = \tau $, the third by $\gamma = \tau+\rho$ and the definition of $T_\gamma$, the fourth by definition of $T_\rho^{-1}$ and the last by \cref{thm:moreau_decomposition}.
\end{proof}
The following important identities are special cases of the previous proposition: Vanilla anisotropic PPA applied to $T_\rho$ is anisotropic PPA with under-relaxation and conversely, anisotropic PPA with under-relaxation is vanilla anisotropic PPA applied to a regularized operator.
\begin{corollary}\label{thm:relaxation_parameter}
Let $T: X \rightrightarrows X^*$ be a set-valued mapping. Then the following identities hold:
\begin{corenum}
\item For all $\rho \geq 0$ we have $\ares{\phi}{T_{\rho}} = \tfrac{\rho}{1+\rho} \id + \tfrac{1}{1+\rho} \ares{(1+\rho)\star \phi}{T}$. \label{thm:relaxation_parameter_1}
\item For all $\lambda \in (0, 1]$ we have $(1 - \lambda) \id + \lambda \ares{\phi}{T} = \ares{\lambda \star \phi}{T_{1-\lambda}}$. \label{thm:relaxation_parameter_2}
\end{corenum}
\end{corollary}
\begin{proof}
``\labelcref{thm:relaxation_parameter_1}'': The desired result follows from \cref{thm:absorb_relaxation_stepsize} for $\tau=1$.

``\labelcref{thm:relaxation_parameter_2}'': Choose $\rho = 1- \lambda$ and $\tau = \lambda$. Hence, $\gamma = \tau + \rho = 1$ and $\tau/\gamma = \lambda$ and we obtain the desired result from \cref{thm:absorb_relaxation_stepsize}.
\end{proof}

Since the Bregman--Yosida regularization $T_{1-\lambda}$ inherits all the necessary properties from $T$, we can harmlessly restrict the under-relaxation parameter $\lambda=1$ in the analysis of the anisotropic proximal point algorithm \cref{eq:update}.
This is also the case for the following coercivity property that is useful in the analysis of the proposed method:
\begin{definition} \label{def:coercive_operator}
    Let $T: X \rightrightarrows X^*$ be a set-valued mapping. We say that $T$ is coercive if
    $$
    \lim_{\|x\| \to \infty} \inf_{y \in T(x)}\|y\|_* = \infty,
    $$
    where $\inf \emptyset := \infty$.
\end{definition}
Clearly, $T$ is coercive if $\dom T$ is bounded.

As we show next, $T_\rho$ inherits the coercivity from $T$:
\begin{lemma}
Let $T: X \rightrightarrows X^*$ be a set-valued mapping. Then $T$ is coercive if and only if $T_\rho$ is coercive.
\end{lemma}
\begin{proof}
Let $\|z^\nu\| \to \infty$ and $v^\nu \in T_\rho(z^\nu) = (\rho \nabla \phi^* + T^{-1})^{-1}(z^\nu)$. This means that $z^\nu = \rho \nabla \phi^*(v^\nu) + x^\nu$ for some $x^\nu \in T^{-1}(v^\nu)$.
Suppose that $\|v^\nu\|_* \not\to \infty$. This means we can find a subsequence $v^{\nu_j} \to v^\infty$. Since $\|z^{\nu_j}\|=\|\rho \nabla \phi^*(v^{\nu_j}) + x^{\nu_j}\| \to \infty$ and $\|z^{\nu_j}\|-\rho\|\nabla \phi^*(v^{\nu_j})\| \leq \|x^{\nu_j}\|$ we have that $\|x^{\nu_j}\| \to \infty$. Since $T$ is coercive we have that $\|v^{\nu_j}\|_* \to \infty$, a contradiction.

To prove the opposite direction assume that $T_\rho$ is coercive and consider $\|x^\nu\| \to \infty$ and $v^\nu \in T(x^\nu)$. Suppose that $\|v^\nu\|_* \not\to \infty$. This means we can find a subsequence $v^{\nu_j} \to v^\infty$. We have that $z^\nu:=\rho \nabla \phi^*(v^\nu) + x^\nu \in (T_\rho)^{-1}(v^\nu)$, and hence $\|z^{\nu_j}\| \to \infty$. By coercivity of $T_\rho$ this means that $\|v^{\nu_j}\|_* \to \infty$, a contradiction.
\end{proof}

\section{Convergence analysis}
\label{sec:convergence}
\subsection{Projective interpretation} \label{sec:projective_interpretation}
In this section we analyze the convergence of algorithm \cref{eq:anisotropic_ppa}. In light of \cref{thm:relaxation_parameter_2}, we will harmlessly restrict $\lambda =1$. Then the variable $z^k$ can be eliminated and the algorithm in \cref{eq:anisotropic_ppa} simplifies to:
\begin{subequations} \label{eq:algorithm_unitrelaxation}
\begin{align}
    v^k &\in T(x^{k+1})\label{eq:algorithm_unitrelaxation_dual} \\
    x^{k+1} &= x^k - \nabla \phi^*(v^k) \label{eq:algorithm_unitrelaxation_primal},
\end{align}
\end{subequations}
or written in terms of the resolvent:
\begin{align}
    x^{k+1} &= \ares{\phi}{T}(x^k).
\end{align}
A key ingredient for understanding the dynamics of the anisotropic PPA is a projective interpretation of the algorithm which is classically invoked to incorporate error tolerances; see, e.g., \cite{solodov1999hybrid}.

At each iteration $k$ we construct the half-space
\begin{align} \label{eq:halfspace}
H_k := \{ w \in X : \langle w - \ares{\phi}{T}(x^k), v^k\rangle \leq 0 \},
\end{align}
whose normal and support vectors, $v^k=\nabla \phi(x^k - \ares{\phi}{T}(x^k))$ and $\ares{\phi}{T}(x^k)$, are readily available upon evaluation of the resolvent operator.

Then $H_k$ strictly separates $x^k$ from the solutions in $\zer T$, 
unless $x^k \in \zer T$. Furthermore, the proximal point $\ares{\phi}{T}(x^k)$ can be interpreted in terms of a certain non-Euclidean projection of the old iterate onto the half-space $H_k$. All of these properties are established in the next proposition:
\begin{proposition} \label{thm:halfspace}
    Let $\{x^k\}_{k=0}^\infty$ be the sequence of iterates produced by \cref{eq:algorithm_unitrelaxation}. Then the following properties hold:
\begin{propenum}
    \item \label{thm:halfspace:solution} $\zer T \subseteq H_k$
    \item \label{thm:halfspace:old} $x^k \not\in H_k$ unless $x^k \in \zer T$
    \item \label{thm:halfspace:projection} $\ares{\phi}{T}(x^k)$ is the \emph{anisotropic projection} of $x^k$ onto $H_k$, i.e., the following identity holds true:
    \begin{align} \label{eq:anisotropic_halfspace_projection}
\ares{\phi}{T}(x^k) = \aproj{H_k}(x^k):= \argmin_{x \in H_k}\,\phi(x^k-x).
\end{align}

\end{propenum}
\end{proposition}
\begin{proof}
``\labelcref{thm:halfspace:solution}'': Let $x^\star \in \zer T$. Since $T$ is monotone
we have that
\begin{align}
\langle x^\star - \ares{\phi}{T}(x^k), v^k - 0\rangle \leq 0,
\end{align}
and hence $x^\star \in H_k$.

``\labelcref{thm:halfspace:old}'': Since $x^{k} - \ares{\phi}{T}(x^k) = \nabla \phi^*(v^k)$ we have by strict monotonicity of $\nabla \phi^*$ and the relation $\nabla \phi^*(0)=0$ that
\begin{align}
\langle x^k - \ares{\phi}{T}(x^k), v^k \rangle = \langle \nabla \phi^*(v^k) - \nabla \phi^*(0), v^k - 0 \rangle > 0,
\end{align}
i.e., $x^k \not\in H_k$, unless $v^k= 0$ or equivalently, $\ares{\phi}{T}(x^k) = x^k$ and consequently $0 \in T(\ares{\phi}{T}(x^k))$ and $x^k \in \zer T$.

``\labelcref{thm:halfspace:projection}'': Firstly, note that $\ares{\phi}{T}(x^k) \in H_k$ since $\langle x^{k+1} - \ares{\phi}{T}(x^k), v^k \rangle = 0$.
Secondly, by construction of $H_k$, since $v^k = \nabla \phi(x^k-\ares{\phi}{T}(x^k))$, $\ares{\phi}{T}(x^k)$ satisfies the first-order optimality condition of the projection which amounts to
\begin{align}
\nabla \phi(x^k-\ares{\phi}{T}(x^k))\in N_{H_k}(\ares{\phi}{T}(x^k)) = \{\lambda v^k : \lambda \geq 0 \},
\end{align}
and hence the third item is proved.

\end{proof}
Unfortunately, the anisotropic projection does not enjoy nonexpansiveness in the Euclidean sense in general as evidenced in \cref{fig:toilette}.
An important exception is the class of isotropic $\phi$ as defined in the next subsection.
\subsection{Isotropic case} \label{sec:projection_interpretation}
\subsubsection{Global convergence}
In this subsection we recapitulate the global convergence of anisotropic PPA under isotropy providing a specialization of the main results of Luque \cite{luque1986nonlinear,luque1987nonlinear} and refine the local convergence analysis invoking a recent result by Rockafellar \cite{rockafellar2021advances}. Luque's PPA takes the form $x^{k+1}=(I + M^{-1} \circ T)^{-1}(x^k)$ for a maximal monotone map $M$ that preserves directions in the sense that $v \in M(y)$ implies that $v = \lambda y$, for some $\lambda >0$. In fact, this assumption seems crucial as there exist choices for $T$ and $M$ in terms of rotators such that the algorithm is nonconvergent \cite{luque1986nonlinear}.
For $M=\nabla \phi$ being the gradient of a Legendre convex function directions are preserved if $\phi$ exhibits isotropy:
\begin{definition}[isotropic $\phi$]
We say that $\phi$ is \emph{isotropic} if it can be written as
\begin{align} \label{eq:isotropic}
\phi(x)=\varphi(\|x\|_2),
\end{align}
for an increasing function $\varphi :[0,+\infty) \to [0, +\infty)$ that vanishes only at zero.
\end{definition}
In the special case of isotropic $\phi$ we refer to the algorithm in \cref{eq:anisotropic_ppa} as \emph{isotropic} PPA. 

In the isotropic case, a prominent choice is $\varphi(t)= \frac{1}{p} t^p$ for some $p >1$.
Not even in this case is $\nabla \phi^* \circ T$ guaranteed to be monotone as shown in \cref{ex:counter_example_monotone}.

Next we leverage the projective interpretation from the previous subsection to show convergence in the isotropic case. The key step is the following lemma which shows that for isotropic $\phi$ the anisotropic projection is essentially Euclidean and thus firmly nonexpansive.
\begin{lemma} \label{thm:isotropic_projection}
Let $\phi \in \Gamma_0(X)$ be isotropic and $C\subseteq X$ nonempty, closed and convex. Then we have for any $y\in X$ that:
\begin{align} \label{eq:anisotropic_euclidean}
\aproj{C}(y)=\argmin_{x \in C}\,\phi(y-x) =\argmin_{x \in C}\,\|y-x\|_2= \proj{C}(y) ,
\end{align}
and in particular is $\aproj{C}$ firmly nonexpansive in the Euclidean sense.
\end{lemma}
\begin{proof}
This readily follows from the fact that $\varphi$ is monotonically increasing and the firm nonexpansiveness of the Euclidean projection.
\end{proof}

Luque analyzed the isotropic PPA under the following error tolerance, which allows the next iterate $x^{k+1}$ to deviate slightly from the exact proximal point $\ares{\phi}{T}(x^k)$:
\begin{align} \label{eq:eps_prox_point}
    \|\ares{\phi}{T}(x^k) - x^{k+1}\|_2 \leq \varepsilon_k,
\end{align}
for $\varepsilon_k \geq 0$ such that the sequence of errors is summable:
$$
\sum_{k=0}^\infty \varepsilon_k < \infty.
$$
Henceforth, we call $x^{k+1}$ an $\varepsilon_k$-proximal point highlighting the analogy to $\varepsilon$-subgradients.
Next we prove the global convergence of isotropic PPA.

\begin{theorem} \label{thm:convergence_isotropic}
Let $\phi$ be isotropic and $T: X \rightrightarrows X^*$ be maximal monotone. Let $\{x^k\}_{k=0}^\infty$ such that \cref{eq:eps_prox_point} holds true and let $\{v^k\}_{k=0}^\infty$ be defined according to \cref{eq:algorithm_unitrelaxation_dual}.
Then the following hold:
\begin{thmenum}
\item \label{thm:convergence_isotropic:fejer} $\{x^k\}_{k=0}^\infty$ is quasi-Fej\'er-monotone relative to the solutions in $\zer T$, i.e., for any $x^\star \in \zer T$:
\begin{equation}
\|x^{k+1} - x^\star\|_2 \leq \|x^k - x^\star\|_2 + \varepsilon_k.
\end{equation}
\item \label{thm:convergence_isotropic:primal}  $\lim_{k \to \infty} x^k = x^\infty \in \zer T$.
\item \label{thm:convergence_isotropic:dual} $\lim_{k \to \infty} v^k =0$ and $\lim_{k \to \infty} \|x^k - x^{k+1}\|_2=0$.
\end{thmenum}
\end{theorem}
\begin{proof}
Let $x^\star \in \zer T$.
Invoking \cref{thm:halfspace:projection} we have that $\ares{\phi}{T}(x^k)=\aproj{H_k}(x^k)$ and since $\phi$ is minimized at $0$, $x^\star = \aproj{H_k} (x^\star)$.
Thanks to \cref{thm:isotropic_projection} is $\aproj{H_k}=\proj{H_k}$ firmly nonexpansive and thus we have that
\begin{align*}
\|x^{k+1} - x^\star\|_2 &=\|x^{k+1} - \ares{\phi}{T}(x^k) + \ares{\phi}{T}(x^k) - x^\star\|_2 \\
&\leq \|x^{k+1} - \ares{\phi}{T}(x^k)\|_2 + \|\proj{H_k} (x^k) - \proj{H_k} (x^\star)\|_2 \\
&\leq \varepsilon_k + \|x^k - x^\star\|_2.
\end{align*}
Since $\{\varepsilon_k\}_{k=0}^\infty$ is summable, $\{x^k\}_{k=0}^\infty$ is quasi-Fej\'er-monotone relative to the solutions in $\zer T$.
The remainder of the proof follows by invoking the arguments in the proof of \cite[Theorem 1]{rockafellar1976monotone}.
\end{proof}
\subsubsection{Local convergence under H\"older growth}
Next we show that in the isotropic case the algorithm converges with higher order under a certain growth condition as established by \cite{luque1986nonlinear,luque1987nonlinear}.
\begin{definition}[Q-order of convergence]
    Let $\{a^k\}_{k=0}^\infty$ be a convergent sequence with limit $L$. Then $\{a^k\}_{k=0}^\infty$ is said to converge to $L$ with order $q \geq 1$ and rate $\mu >0$ where $\mu < 1$ if $q=1$, if
    \begin{align}
        \limsup_{k \to \infty} \frac{\|a^{k+1} -L\|}{\|a^k-L\|^q} = \mu.
    \end{align}
    In particular when $q=1$ we say that $\{a^k\}_{k=0}^\infty$ converges Q-linearly to $L$ with rate $\mu$ and Q-superlinearly if $q>1$.
\end{definition}
We assume that the following growth property that is implied by a H\"older version of metric subregularity \cite[Definition 3.1]{dontchev2004regularity} at $x^\star \in \zer T$ for $0$, holds true:
\begin{align} \label{eq:growth_euclidean}
\exists \delta,\nu, \rho > 0 : x \in T^{-1}(v), \|x - x^\star\|_2 \leq \delta, \|v\|_2 \leq \delta \Rightarrow \dist{\zer T}{x} \leq \rho \|v\|_2^\nu.
\end{align}

Note that Luque \cite{luque1986nonlinear} did not require the points $x$ to be near the solution which renders our condition \cref{eq:growth_euclidean} to be slightly weaker. Regardless of this relaxation, the proof stays unaffected as nearness of $x^k$ to a solution is guaranteed by \cref{thm:convergence_isotropic:primal}.
\begin{example}[running example] \label{ex:growth_euclidean}
Choose \begin{align}
    T(x) = A x - b \quad\text{for}\quad A= \begin{bmatrix} 0 & -\alpha \\ \alpha & 0\end{bmatrix} \quad\text{and}\quad b  \in \bR^2,\alpha>0.
\end{align}
Then
$$
T^{-1}(0)=A^{-1}b =: x^\star\quad \text{for}\quad A^{-1}=\begin{bmatrix} 0 & \tfrac{1}{\alpha} \\ -\tfrac{1}{\alpha} & 0\end{bmatrix}
$$
is single-valued and hence $\dist{\zer T}{x}=\|x - x^\star\|_2$. Since $\|Ax\|_2=\alpha\|x\|_2$ and $Ax^\star = b$ we obtain
$$
\|x-x^\star\|_2=\tfrac{1}{\alpha}\|Ax - b\|_2 = \tfrac{1}{\alpha}\|T(x)\|_2.
$$
Hence, \cref{eq:growth_euclidean} holds for $\delta=\infty$, $\rho=\tfrac{1}{\alpha}$ and $\nu =1$.
\end{example}

Furthermore, we restrict
\begin{align}
\phi(x)=\tfrac{1}{p}\|x\|_2^p,
\end{align}
for $p> 1.$
Then we have that $\phi^*(v)=\tfrac{1}{q}\|v\|_2^q$ for $\frac{1}{p}+\frac{1}{q}=1$ and $\nabla \phi(x)=\|x\|_2^{p-2} x$ for $x \neq 0$ and $\nabla \phi(0)=0$.
\begin{remark}[equivalence to the classical PPA with implicit step-size] \label{rem:implicite_step_size}
Since $\nabla \phi^*(v^k) = \|v^k\|_2^{q-2} v^k$, the algorithm can be interpreted as classical PPA, $x^{k+1}=(I + c_k T)^{-1}(x^k)$ with an implicitly defined step-size $c_k = \|v^k\|_2^{q-2}$, that depends on future information, namely the dual vector $v^k$. Specifically, if $q < 2$, since $v^k \to 0$, this means that $c_k \to \infty$. Hence, $x^k$ converges superlinearly to $x^\infty$ which is a direct consequence of \cite{rockafellar2021advances}.
\end{remark}
Utilizing the explicit growth of the operator and the power of the prox-function we can further specify the order of Q-superlinear convergence refining \cite{rockafellar2021advances}.

For that purpose, the error bound for the $x$-update \cref{eq:eps_prox_point} has to be tightened:
\begin{align} \label{eq:error_bound_isotropic}
\|x^{k+1} - \ares{\phi}{T}(x^k)\|_2 \leq \varepsilon_k \min\{1, \|x^{k+1} - x^k\|_2^r\},
\end{align}
where $r\geq 1$.

The following result is a special case of \cite[Theorem 3.1]{luque1986nonlinear}:
\begin{proposition} \label{thm:convergence_rate_euclidean}
Assume that the elements of the sequence $\{x^k\}_{k=0}^\infty$ satisfy the error bound \cref{eq:error_bound_isotropic}. Denote by $x^\infty$ its limit point.
    Suppose that $T$ satisfies the growth condition \cref{eq:growth_euclidean} with $\delta, \rho,\nu>0$ at $x^\infty \in \zer T$.
Then the following is true:
    \begin{propenum}
        \item\label{thm:convergence_rate_euclidean:linear} Let $p=\frac{1}{\nu}  + 1$ and $r \geq 1$. Then $\{ \dist{\zer T}{x^k}\}_{k=0}^\infty$ converges  Q-linearly to $0$  with rate $\tfrac{\rho}{\sqrt{1+\rho^2}}$.
        \item\label{thm:convergence_rate_euclidean:superlinear} Let $p>\frac{1}{\nu} + 1$ and $r >1$. Define $s:=\nu(p-1)$. Then $\{ \dist{\zer T}{x^k}\}_{k=0}^\infty$ converges to $0$ with order $\min\{s, r\}$.
    \end{propenum}
\end{proposition}
\begin{proof}
Let $S:=\nabla \phi$. Note that $\|S(x)\|=\|x\|_2^{p-1}$. Hence the result follows by invoking \cite[Theorem 3.1]{luque1986nonlinear} for $\tau(t):=\rho t^\nu$ and $\sigma(t):=t^{p-1}$.
\end{proof}
Note that for $p=2$ and $\nu=1$ we recover the known linear convergence result as first proved in \cite{luque1984asymptotic}. The convergence behaviour is depicted in \cref{fig:toilette} using the problem setup from \cref{ex:growth_euclidean}.

It is rather standard to show that the Q-superlinear convergence of $\{\dist{\zer T}{x^k}\}_{k=0}^\infty$ with higher order implies that $x^k$ converges to $x^\infty$ R-superlinearly with the same order. By adapting a novel technique \cite{rockafellar2021advances} this can be sharpened to Q-superlinear convergence. For that purpose we prove the following result adapted from \cite[Theorem 2.3]{rockafellar2021advances} which shows that the sequences $\{\dist{\zer T}{x^k}\}_{k=0}^\infty$ and $\{\|x^k - x^\infty\|_2\}_{k=0}^\infty$ are asymptotically equivalent:
\begin{proposition} \label{thm:equivalence_order}
Let $\{x^k\}_{k=0}^\infty$ be the sequence of iterates that satisfy the error bound \cref{eq:error_bound_isotropic} with $r=2$ and let $x^\infty$ denote its limit point. Then the ratio of the sequences $\{\dist{\zer T}{x^k}\}_{k=0}^\infty$ and $\{\|x^k - x^\infty\|_2\}_{k=0}^\infty$ approaches $1$:
\begin{equation}
\frac{\dist{\zer T}{x^k}}{\|x^k - x^\infty\|_2} \to 1,
\end{equation}
as $k \to \infty$.
\end{proposition}
\begin{proof}
Using \cref{thm:convergence_isotropic} the proof is identical to the first part of the proof of \cite[Theorem 2.3]{rockafellar2021advances}.
\end{proof}
Based on the above result we obtain that $x^k$ converges to $x^\infty$ with higher order:
\begin{corollary}
    \label{thm:convergence_rate_euclidean_terry}
Assume that the elements of the sequence $\{x^k\}_{k=0}^\infty$ satisfy the error bound \cref{eq:error_bound_isotropic} with $r \geq 2$ and denote its limit point by $x^\infty$.
    Suppose that $T$ satisfies the growth condition \cref{eq:growth_euclidean} with $\delta, \rho,\nu>0$ at $x^\infty \in \zer T$. Then the following is true:
    \begin{corenum}
        \item\label{thm:convergence_rate_euclidean_terry:linear} Let $p=\frac{1}{\nu}  + 1$. Then $\{x^k\}_{k=0}^\infty$ converges Q-linearly to $x^\infty$ with rate $\tfrac{\rho}{\sqrt{1+\rho^2}}$.
        \item\label{thm:convergence_rate_euclidean_terry:superlinear} Let $p>\frac{1}{\nu} + 1$. Define $s:=\nu(p-1)$. Then $\{x^k\}_{k=0}^\infty$ converges to $x^\infty$ with order $\min\{s, r\}$.
    \end{corenum}
\end{corollary}
\begin{proof}

We have
\begin{align}\label{eq:limitsQ}
    \limsup_{k\to \infty}\frac{\|x^{k+1}-x^\infty\|}{\|x^{k}-x^\infty\|} &= \lim_{k\to \infty} \frac{\dist{\zer T}{x^{k+1}}}{\|x^{k+1} - x^\infty\|}  \limsup_{k\to \infty}\frac{\|x^{k+1}-x^\infty\|}{\|x^{k}-x^\infty\|}\nonumber \\
    &= \limsup_{k\to \infty} \frac{\dist{\zer T}{x^{k+1}}}{\|x^{k+1} - x^\infty\|}\frac{\|x^{k+1}-x^\infty\|}{\|x^{k}-x^\infty\|}\nonumber\\
    &= \limsup_{k\to \infty} \frac{\dist{\zer T}{x^{k+1}}}{\|x^{k}-x^\infty\|}\nonumber\\
    &=\limsup_{k\to \infty} \frac{\dist{\zer T}{x^{k+1}}}{\dist{\zer T}{x^{k}}}\frac{\dist{\zer T}{x^{k}}}{\|x^{k}-x^\infty\|}\nonumber\\
    &= \limsup_{k\to \infty} \frac{\dist{\zer T}{x^{k+1}}}{\dist{\zer T}{x^{k}}}\lim_{k\to \infty}\frac{\dist{\zer T}{x^{k}}}{\|x^{k}-x^\infty\|}\nonumber\\
    &= \limsup_{k\to \infty} \frac{\dist{\zer T}{x^{k+1}}}{\dist{\zer T}{x^{k}}}
\end{align}
where the first and last equalities follow from~\Cref{thm:equivalence_order}, and the second and fifth equalities follow from the fact that for two positive sequences $\{\alpha_k\}_{k=0}^\infty$, $\{\beta_k\}_{k=0}^\infty$, such that $\lim_{k\to\infty}\alpha_k$ exists, one has $\limsup_{k\to\infty}\alpha_k\beta_k=\lim_{k\to\infty}\alpha_k\limsup_{k\to\infty}\beta_k$. The claimed rates for $\{x^k\}_{k=0}^\infty$ readily follow from~\eqref{eq:limitsQ} and~\cref{thm:convergence_rate_euclidean}.

\end{proof}

\subsection{Anisotropic case} \label{sec:convergence_anisotropic}
\subsubsection{Global convergence}
The class of isotropic prox-functions $\phi$, discussed in the previous subsection, are nonseparable unless $\varphi(t)=t^2$. However, in the context of minimax optimization, separability, at least between decision and adversarial (multiplier) variables, is desirable in practice; see \cref{sec:alm}.
Hence, in this subsection we study the convergence of the anisotropic PPA for general $\phi$ harnessing the gradient structure of the preconditioner $S=\nabla \phi$ to bypass the isotropy assumption in Luque's PPA.

Unfortunately, in contrast to the Euclidean projection the anisotropic projection does not enjoy firm nonexpansiveness. As we shall demonstrate in \cref{rem:nontelescopable}, standard techniques for the classical or the Bregman PPA for monotone operators are not applicable in the anisotropic setting.
In the same way that isotropic PPA can be seen as classical PPA with implicitly defined step-sizes, anisotropic PPA appears similar in form to a variable metric PPA with implicitly defined metric. However, the assumptions on the metric imposed in \cite{rockafellar2021advances} are too strong to cover this case.
Instead, if $T=\partial f$ is the subdifferential of a convex function, one can utilize $f$ or the value-function of the anisotropic PPA update, i.e., the infimal convolution of $f$ and $\phi$, as Lyapunov functions which is explored in \cite{oikonomidis2025global} for the special case $\phi=\frac{1}{p}\|\cdot\|_p^p$ with $p>2$.
For general monotone operators $T$ we will track the algorithm's progress in the dual space using Bregman distances. This enables us to deduce its primal subsequential convergence indirectly as stated in our main result \cref{thm:convergence_global_anisotropic}. For that purpose we define the Bregman distance $D_{\phi^*} : X^* \times X^* \to \bR$ generated by the conjugate prox-function $\phi^*: X^* \to \bR$ as
\begin{align}
D_{\phi^*}(u,v) = \phi^*(u) - \phi^*(v) - \langle \nabla \phi^*(v), u-v\rangle.
\end{align}
Thanks to our standing assumptions on $\phi$, the Bregman distance $D_{\phi^*}$ enjoys the following favorable properties:
\begin{lemma} \label{thm:props_bregman_dist}
The following properties hold:
\begin{lemenum}
    \item \label{thm:props_bregman_dist:coercive} $D_{\phi^*}(u, \cdot)$ is coercive for every $u \in X^*$.
    \item \label{thm:props_bregman_dist:consitency} If $\lim_{k \to \infty} v^k = v$ we have $\lim_{k \to \infty} D_{\phi^*}(v^k, v) =0$.
    \item \label{thm:props_bregman_dist:convergence} Let $\{v^k\}_{k=0}^\infty$ and $\{u^k\}_{k=0}^\infty$ be sequences in $X^*$ and assume that one of the two sequences converges to some limit. Then $\lim_{k \to \infty} D_{\phi^*}(u^k, v^k) = 0$ implies that the other sequence converges to the same limit.
    \item \label{thm:props_bregman_dist:dual} For any $x,y \in X$ we have $D_{\phi^*}(\nabla \phi(x), \nabla \phi(y))=D_{\phi}(y, x)$.
\end{lemenum}
\end{lemma}
\begin{proof}
``\labelcref{thm:props_bregman_dist:coercive,thm:props_bregman_dist:consitency}'': Since $\phi$ is super-coercive with full domain, via \cite[Theorem 5.6]{bauschke1997legendre}, $\phi^*$ is Bregman/Legendre in the sense of \cite[Definition 5.2]{bauschke1997legendre}. In particular, in light of \cite[Remark 5.3]{bauschke1997legendre}, $D_{\phi^*}(v, \cdot)$ is coercive and $\lim_{k \to \infty} v^k = v \in X^*$, implies $\lim_{k \to \infty} D_{\phi^*}(v^k, v) = 0$.

``\labelcref{thm:props_bregman_dist:convergence}'': This is a consequence of \cite[Theorem 2.4]{solodov2000inexact}.

``\labelcref{thm:props_bregman_dist:dual}'': This follows from \cite[Theorem 3.7(v)]{bauschke1997legendre}.
\end{proof}

We state the celebrated three-point identity \cite{chen1993convergence} which is key to establish the convergence of the anisotropic PPA in the dual space. Classically, it is used to establish convergence of the Bregman PPA in the primal space \cite{chen1993convergence,Eckstein93,solodov2000inexact}:
\begin{lemma}[Three-point identity] \label{thm:threepointlemma}
For any $u,v,w \in X^*$ the following identity holds true:
\begin{align}
D_{\phi^*}(w, v) = D_{\phi^*}(u, v) + D_{\phi^*}(w, u) + \langle \nabla \phi^*(u) - \nabla \phi^*(v),w - u \rangle.
\end{align}
\end{lemma}
\begin{remark}[inapplicability of standard techniques] \label{rem:nontelescopable}
    In the classical and the Bregman PPA the monotonicity of $T$ between the pairs $(x^{k+1}, v^k)$ and $(x^\star, 0)$ is utilized to establish (a Bregman version of) Fej\'er monotonicity of the primal iterates $\{x^k\}_{k=0}^\infty$ wrt. to the solution set $\zer T$ which ultimately ensures convergence. It is tempting to apply the same pattern in the anisotropic case:
    By \cref{eq:algorithm_unitrelaxation_primal} one has $v^k = \nabla \phi(x^k - x^{k+1})$ and since by assumption, $\nabla \phi(0)=0$, monotonicity of $T$ yields:
    \begin{align}
        \langle x^k-x^\star -(x^k-x^{k+1}) ,\nabla \phi(x^k - x^{k+1}) - \nabla \phi(0)\rangle \geq 0.
    \end{align}
    Using the identity transformation in \cref{thm:threepointlemma} with $D_\phi$ and $u =x^k - x^{k+1}$, $v=0$ and $w=x^k-x^\star$ we arrive to:
    \begin{align}
        D_{\phi}(x^k-x^\star, 0) \geq D_{\phi}(x^k - x^{k+1}, 0) + D_{\phi}(x^k-x^\star, x^k - x^{k+1}),
    \end{align}
    which can be rewritten using $\nabla \phi(0)=0$:
    \begin{align} \label{eq:primal_fejer}
        D_{\phi}(x^k-x^\star, x^k - x^{k+1}) + \phi(x^k -x^{k+1}) \leq \phi(x^k - x^\star).
    \end{align}
    Unfortunately, this inequality cannot be telescoped since $D_\phi$ is not translation invariant wrt. $x^k$.
    If, furthermore, $\phi(x)=\frac{1}{p}\|x\|_p^p$ for $p\geq 2$, $\phi$ is uniformly convex relative to itself and hence, in light of \cite[Lemma 4.2.3]{nesterov2018lectures}, we have $D_\phi(x,y) \geq \tfrac{1}{2^{p-2}}\phi(x-y)$. In this case we can further bound \cref{eq:primal_fejer}:
    $$
        \phi(x^{k+1}-x^\star) + 2^{p-2}\phi(x^k -x^{k+1}) \leq 2^{p-2}\phi(x^k - x^\star).
    $$
    This inequality can be telescoped only if $p=2$ which recovers the known Euclidean case.
\end{remark}
As standard techniques fail to derive convergence of anisotropic PPA \cref{eq:algorithm_unitrelaxation} we propose to establish a Bregman version of Fej\'er monotonicity in the dual space.
To this end we define the sequence
\begin{align}
y^{k}:= x^{k} - x^{k+1},
\end{align}
and hence we have that $v^{k}=\nabla \phi(y^{k})$.

\begin{lemma}[dual space Bregman--Fej\'er-monotonicity] \label{thm:dual_space_fejer} The following conditions hold true for the dual iterates $\{v^k\}_{k=0}^\infty$ generated by the algorithm in \cref{eq:algorithm_unitrelaxation}:
\begin{lemenum}
    \item \label{thm:dual_space_fejer:decr} We have that
$D_{\phi^*}(v^{k+1}, 0) \leq D_{\phi^*}(v^{k}, 0) -D_{\phi^*}(v^k, v^{k+1}) 
$.
    \item \label{thm:dual_space_fejer:convergence} $\{D_{\phi^*}(v^k, 0)\}_{k=0}^\infty$ converges.
    \item \label{thm:dual_space_fejer:sum} $\{D_{\phi^*}(v^k, v^{k+1})\}_{k=0}^\infty$ is summable and in particular $\lim_{k \to \infty} D_{\phi^*}(v^{k}, v^{k+1}) = 0$.
    \item \label{thm:dual_space_fejer:identity_dual} It holds that $D_{\phi^*}(v^k, 0)=D_{\phi}(0, y^k)$ and $D_{\phi^*}(v^k, v^{k+1})=D_{\phi}(y^{k+1}, y^k)$.
    \item \label{thm:dual_space_fejer:sum_inner} We have that
    $$
    0 \leq \sum_{k=0}^\infty \langle x^{k+2} -x^{k+1},v^{k+1} - v^k \rangle 
    < \infty,$$ and hence
$\lim_{k \to \infty}\langle x^{k+2} - x^{k+1},v^{k+1} -v^k \rangle = 0$.
    \item \label{thm:dual_space_fejer:bounded} $\{v^k\}_{k=0}^\infty$ and $\{y^k\}_{k=0}^\infty$ are bounded.
\end{lemenum}
\end{lemma}
\begin{proof}
``\labelcref{thm:dual_space_fejer:decr}'': Thanks to \cref{thm:threepointlemma} we obtain for $w = v^k$, $u = v^{k+1}$ and $v = 0$ that:
\begin{align}
D_{\phi^*}(v^k, 0) = D_{\phi^*}(v^{k+1}, 0) + D_{\phi^*}(v^k, v^{k+1}) + \langle \nabla \phi^*(v^{k+1}) - \nabla \phi^*(0), v^k - v^{k+1} \rangle.
\end{align}
Reordering the equality and using the fact that $\nabla \phi^*(v^{k+1}) = x^{k+1} - x^{k+2}$ \cref{eq:algorithm_unitrelaxation_primal} and $\nabla \phi^*(0)=0$ we arrive to
\begin{align} \label{eq:3point_fejer}
D_{\phi^*}(v^{k+1}, 0) =  D_{\phi^*}(v^k, 0) - D_{\phi^*}(v^k, v^{k+1}) - \langle x^{k+2} -x^{k+1},v^{k+1} - v^k \rangle.
\end{align}
Thanks to the update \cref{eq:algorithm_unitrelaxation_dual} we have that
$
v^{k+1} \in T(x^{k+2})$ and $v^k \in T(x^{k+1}).
$
Monotonicity of $T$ yields the inequality
\begin{align} \label{eq:monotone_iter}
\langle x^{k+2}-x^{k+1},v^{k+1} -v^k \rangle \geq 0,
\end{align}
and hence we obtain
\begin{align} \label{eq:dual_fejer}
D_{\phi^*}(v^{k+1}, 0) \leq D_{\phi^*}(v^{k}, 0) -D_{\phi^*}(v^k, v^{k+1}),
\end{align}
and the first item is proved.

``\labelcref{thm:dual_space_fejer:convergence,thm:dual_space_fejer:sum}'':

Invoking \cite[Lemma 5.31]{BaCo110} with \cref{eq:dual_fejer} we obtain that $\{D_{\phi^*}(v^k, 0)\}_{k=0}^\infty$ converges and
\begin{align}
\sum_{k=0}^\infty D_{\phi^*}(v^{k}, v^{k+1}) < \infty,
\end{align}
implying that $\lim_{k \to \infty} D_{\phi^*}(v^k, v^{k+1}) = 0$.

``\labelcref{thm:dual_space_fejer:identity_dual}'': Invoking \cref{thm:props_bregman_dist:dual}, using the relations $v^{k} = \nabla \phi(y^k)$ and $0 =\nabla \phi(0)$ we have the identities
\begin{align}
 D_{\phi^*}(v^{k}, 0)=D_\phi(0, y^{k}) \qquad \text{and} \qquad D_{\phi^*}(v^k, v^{k+1})=D_{\phi}(y^{k+1}, y^{k}).
\end{align}

``\labelcref{thm:dual_space_fejer:sum_inner}'': Summing the estimate \cref{eq:3point_fejer} yields via \cref{eq:monotone_iter} and the nonnegativity of the Bregman distance that:
\begin{align*}
0 \leq \sum_{k=0}^\infty \langle x^{k+2} -x^{k+1},v^{k+1} - v^k \rangle 
\leq D_{\phi^*}(v^0, 0) < \infty,
\end{align*}
and hence
$$
\lim_{k \to \infty} \langle x^{k+2} - x^{k+1},v^{k+1} - v^k \rangle 
= 0.
$$

``\labelcref{thm:dual_space_fejer:bounded}'': Combining \labelcref{thm:dual_space_fejer:convergence,thm:dual_space_fejer:identity_dual} we obtain that $\{D_\phi(0, y^k)\}_{k=0}^\infty$ converges. Thanks to \cref{thm:props_bregman_dist:coercive} and the fact that $\phi$ and $\phi^*$ are both super-coercive and Legendre with full domain, $D_\phi(0, \cdot)$ is coercive implying that $\{y^k\}_{k=0}^\infty$ is bounded. Via the identity $v^k=\nabla \phi(y^k)$ and the continuity of $\nabla \phi$ the same is true for $\{v^k\}_{k=0}^\infty$.

\end{proof}
The result can be directly extended to accommodate summable error tolerances by employing the $\varepsilon$-enlargement of monotone operators \cite{burachik1997enlargement}.

The inequality $D_{\phi^*}(v^{k+1}, 0) \leq D_{\phi^*}(v^{k}, 0)$ can be understood in terms of dual space Fej\'er monotonicity: In each step the Bregman distance between the dual iterates $v^k$ and $0 \in T(x^\star)$ decreases.

This readily implies convergence in the following special case where $T$ is defined as in \cref{ex:growth_euclidean}. Notably, this renders a boundary case for monotonicity:
\begin{example}[primal Fej\'er-monotonicity] \label{ex:example_skew}
    Let $T$ be defined as in \cref{ex:growth_euclidean} and choose $\phi \in \Gamma_0(\bR^2)$ as $\phi(x):=\frac{1}{p}\|x\|_p^p$ with $p>2$. Then $\phi^*(x) =\frac{1}{q}\|x\|_q^q$ for $q=\frac{p}{p-1}$. Then we have that
$$
\tfrac{1}{q}\|x^{k+1}-x^\star\|_q^q \leq \tfrac{1}{q}\|x^{k}-x^\star\|_q^q  -D_{\phi^*}(x^k-x^\star, x^{k+1}-x^\star).
$$
\end{example}

\begin{figure}[!t]
\centering
\subfloat[\label{fig:toilette:anisotropic_iterates}]{
        \centering
        \resizebox{0.47\textwidth}{!}{
\begin{tikzpicture}

\definecolor{darkgray176}{RGB}{176,176,176}
\definecolor{lightgray204}{RGB}{204,204,204}

\begin{axis}[
axis equal,
legend cell align={left},
legend style={fill opacity=0.8, draw opacity=1, text opacity=1, draw=lightgray204},
legend style={font=\scriptsize, legend cell align=left, align=left, draw=white!15!black},
tick align=outside,
tick pos=left,
x grid style={darkgray176},
xmin=-5.46195090107524, xmax=8.51247285730578,
xtick style={color=black},
xticklabel style={font=\scriptsize},
y grid style={darkgray176},
ymin=-7.10314195516781, ymax=6.25241640983784,
ytick style={color=black},
yticklabel style={font=\scriptsize}
]
\addplot [line width=1pt, black, mark=*, mark options={solid}]
table {%
-4.82674982114883 1.92558928253914
-3.1635628854488 3.5323815661665
-1.31691818334733 4.8201933115989
0.638247484194027 5.64534557506486
2.52665813691733 5.13218958663769
4.2683822108874 4.0672054996938
5.79953931913493 2.68888418027388
7.04466947122497 1.10069819128755
7.87727177737937 -0.613546799572763
7.507790246681 -2.27303320305442
6.56344640270229 -3.78357059132098
5.32394145558681 -5.07274502784747
3.90131001757354 -6.04776041684745
2.4019649738303 -6.49607112039483
1.02401257058288 -5.79750565123066
-0.150824603874787 -4.76048477297536
-1.02561659378611 -3.53052205122601
-1.36696097874637 -2.23303197828781
-0.686533919931151 -1.07403803526413
0.276353975002298 -0.145693803581262
1.37485066821728 0.413389970008493
2.37057625657883 -0.0170611053640912
3.15279445061107 -0.776269393849888
3.5646444249167 -1.6607591973289
3.11356850569581 -2.40693896980194
2.44980064748504 -2.88117553918743
1.88397488774177 -2.64031758077809
1.58155867663109 -2.18291112948509
1.82188523321859 -1.88448629239758
2.00473306993866 -1.93313333721358
2.00893286835559 -1.99996472338651
2.00000007594565 -2.00015958956048
};
\addlegendentry{$x^k$ (anisotropic)}
\end{axis}

\end{tikzpicture}
 	   }
  }
  \subfloat[]{
        \centering
        \resizebox{0.47\textwidth}{!}{
		\input{figures/toilette_bregman}
 	    }
  }\\
  \subfloat[\label{fig:toilette:anisotropic_fejer}]{
        \centering
       \resizebox{0.47\textwidth}{!}{
\begin{tikzpicture}

\definecolor{darkgray176}{RGB}{176,176,176}
\definecolor{lightgray204}{RGB}{204,204,204}

\begin{axis}[
legend cell align={left},
legend style={fill opacity=0.8, draw opacity=1, text opacity=1, draw=lightgray204},
legend style={font=\scriptsize, legend cell align=left, align=left, draw=white!15!black},
log basis y={10},
tick align=outside,
tick pos=left,
x grid style={darkgray176},
xmajorgrids,
xmin=-1.2, xmax=25.2,
xtick style={color=black},
xticklabel style={font=\scriptsize},
y grid style={darkgray176},
ymajorgrids,
ymin=1.07320644274344, ymax=9.59945497683632,
ymode=log,
ytick style={color=black},
yticklabel style={font=\scriptsize}
]
\addplot [line width=1pt, black, dashed, mark=x, mark repeat=5, mark options={solid}]
table {%
0 7.87494535446068
1 7.56766988350068
2 7.58399519004321
3 7.76567311161548
4 7.15160800749903
5 6.47738685197859
6 6.03507530140884
7 5.92140349472683
8 6.0385905658651
9 5.51455352058573
10 4.89960887465127
11 4.52662664755223
12 4.4720626309481
13 4.51400391668359
14 3.92091834058158
15 3.49947462606345
16 3.3907010369366
17 3.37501557556487
18 2.84163158468478
19 2.53168068475614
20 2.49304288658779
21 2.01726880256446
22 1.68120541332348
23 1.60099859419671
24 1.18559442560306
};
\addlegendentry{$\|x^k- x^\star\|_2$}
\addplot [line width=1pt, black, mark=diamond*, mark repeat=5, mark options={solid}]
table {%
0 8.6894782106672
1 8.49190023750256
2 8.28613649934234
3 8.02383903463324
4 7.2272825742234
5 6.95981316396125
6 6.75575509870038
7 6.55704778989581
8 6.31803847937929
9 5.5482429463391
10 5.27941561253688
11 5.07901077673075
12 4.87550660601314
13 4.57584449319727
14 4.12058556754283
15 3.91312562837502
16 3.71358409059667
17 3.40770872690576
18 3.03770813815702
19 2.84076933905535
20 2.6210976567304
21 2.08834998279299
22 1.88666932326146
23 1.66825600299882
24 1.27205655005412
};
\addlegendentry{$\|x^k- x^\star\|_q$}
\end{axis}

\end{tikzpicture}
 	 }
  }
  \subfloat[]{
        \centering
        \resizebox{0.47\textwidth}{!}{
	    \input{figures/toilette_convergence_comparison}
 	 }
  }
\caption{Convergence comparison of the anisotropic and Bregman PPA for $\phi=\frac{1}{3}\|\cdot\|_{3}^{3}$ applied to $T$ as in \cref{ex:growth_euclidean} with $\alpha=\tfrac{1}{2}$, $b=(1,1)$ and hence $x^\star=(2,-2)$. In the upper row we display the iterates of (a) the anisotropic and (b) the Bregman PPA. Subfigure (c) compares the distance to the solution in the $\ell^2$ and $\ell^q$ sense for the first 25 iterations of anisotropic PPA.
It can be seen that the sequence of iterates $\{x^k\}_{k=0}^\infty$ is not Fej\'er-monotone in the $\ell^2$ sense but in the $\ell^q$ sense as a shown in \cref{ex:example_skew}.
Subfigure (d) compares Euclidean, Bregman and anisotropic PPA. In light of \cref{thm:convergence_rate_anisotropic,ex:growth_anisotropic} anisotropic PPA converges cubically whereas by \cref{thm:convergence_rate_euclidean,ex:growth_euclidean} Euclidean PPA converges linearly with rate $\frac{2}{\sqrt{5}}$. Bregman PPA converges poorly unless the Bregman kernel is ``centered'' at the a-priori unknown solution $x^\star$.}
\label{fig:toilette}
\end{figure}

\begin{proof}
It holds that $T(x)=Ax-b=A(x -x^\star)$ and $\phi^*(Ax) = \phi^*(-\alpha x_2, \alpha x_1)=\alpha^q\frac{1}{q}\|x\|_q^q$. Hence $D_{\phi^*}(T(x), 0) = \phi^*(A(x-x^\star))=\alpha^q\frac{1}{q}\|x-x^\star\|_q^q$. Furthermore we have that $D_{\phi^*}(T (x^k), T(x^{k+1})) = \alpha^qD_{\phi^*}(x^k-x^\star, x^{k+1}-x^\star)$. Then \cref{thm:dual_space_fejer:decr} reads
$$
\tfrac{1}{q}\|x^{k+1}-x^\star\|_q^q \leq \tfrac{1}{q}\|x^{k}-x^\star\|_q^q  -D_{\phi^*}(x^k-x^\star, x^{k+1}-x^\star).
$$
The situation is illustrated in \cref{fig:toilette:anisotropic_iterates,fig:toilette:anisotropic_fejer}.
\end{proof}

The next result shows that dual space Fej\'er monotonicity readily implies convergence of the full sequence if $T$ is uniformly monotone:
\begin{definition}
    Let $T: X \rightrightarrows X^*$ be a set-valued mapping. We say that $T$ is uniformly monotone with modulus $\omega:[0,\infty) \to [0,\infty)$ if $\omega$ is increasing and vanishes only at 0 such that for all $(x,u), (y,v) \in \gph T$:
    \begin{align}
        \langle x-y, u-v\rangle \geq \omega(\|x-y\|_2).
    \end{align}
\end{definition}

\begin{proposition} \label{thm:uniform_monotonicity}
Assume that $T$ is uniformly monotone with modulus $\omega$. Then the following hold:
\begin{propenum}
    \item \label{thm:uniform_monotonicity:sum} $\sum_{k=1}^\infty \omega(\|x^{k+1} - x^k\|_2) < \infty$
    \item \label{thm:uniform_monotonicity:convergence} $\lim_{k \to \infty} x^k = x^\star = T^{-1}(0)$.
\end{propenum}
\end{proposition}
\begin{proof}
In light of \cref{thm:dual_space_fejer:sum_inner}, using the uniform monotonicity of $T$, since $v^{k+1} \in T(x^{k+2})$ and $v^k \in T(x^{k+1})$ we have that
$$
\sum_{k=0}^\infty \omega(\|x^{k+2} - x^{k+1}\|_2) \leq \sum_{k=0}^\infty \langle x^{k+2} - x^{k+1},v^{k+1} - v^k \rangle < \infty.
$$

Since $\{\omega(\|x^{k+1} - x^k\|_2)\}_{k =1}^\infty$ is summable we have that $\omega(\|x^{k+1} - x^k\|_2) \to 0$. Suppose that $\|x^{k+1} - x^k\|_2 \not\to 0$. This means that there is a subsequence indexed by $k_j$ along which $\|x^{k_j+1} - x^{k_j}\|_2 \geq \varepsilon >0$. Since $\omega$ is increasing and vanishes only at $0$ we have that $\omega(\|x^{k_j+1} - x^{k_j}\|_2) \geq \omega(\varepsilon)>0$, a contradiction. Hence,
$v^k = \nabla \phi(x^{k} - x^{k+1}) \to 0$. In light of \cite[Theorem 4.5]{liu2023strongly}, $T^{-1}$ is uniformly continuous and $T$ has a unique zero $x^\star$, i.e., $T^{-1}(0)=x^\star$, and hence $x^{k+1} = T^{-1}(v^k) \to T^{-1}(0)=x^\star$.
\end{proof}

In general, dual space Fej\'er monotonicity is a rather weak property. Nonetheless, it can be proved that the primal iterates cannot escape indefinitely in the same direction: As shown in \cref{thm:halfspace}, at each iteration $k$, the algorithm generates a separating hyperplane, supported by $x^{k+1}$ and spanned by normal vector $v^k$, that separates the current iterate $x^k$ from any solution. Consequently, if $v^k \to v^\star$ converges, the separating hyperplanes are parallel and equidistant in the limit contradicting the separation of the current iterate $x^k$ from a solution. This is illustrated in \cref{fig:visualization} and proved in the following standalone result:

\begin{proposition} \label{thm:standalone}
    Let $T: X \rightrightarrows X^*$ be monotone.
Suppose that $\lim_{k \to \infty} v^k = v^\star$. Then necessarily $v^\star = 0$.
\end{proposition}
\begin{proof}
Since $x^{k+1}=x^{k} - \nabla \phi^*(v^{k})$ we have for any $N \in \bN$ that $x^{k+1+N}= x^{k} - \sum_{i=0}^{N} \nabla \phi^*(v^{k+i})$. Thus we obtain using the relation $\nabla \phi^*(0)=0$
\begin{align}
\langle x^\star - x^{k+1+N}, v^{k+N} \rangle
&= \langle x^\star-x^k, v^{k+N} \rangle + \sum_{i=0}^{N} \langle \nabla \phi^*(v^{k+i}), v^{k+N} \rangle \notag \\
&= \langle x^\star-x^k, v^{k+N} \rangle + \sum_{i=0}^{N} \langle \nabla \phi^*(v^{k+i}) - \nabla \phi^*(0), v^{k+N} - 0 \rangle. \label{eq:sum_lemma}
\end{align}
Assume that the contrary holds, i.e., $v^\star \neq 0$. Then by strict monotonicity of $\nabla \phi^*$ we have that
$$
\langle \nabla \phi^*(v^{\star}) - \nabla \phi^*(0), v^{\star} - 0 \rangle = 2\varepsilon > 0,
$$
for some $\varepsilon > 0$.
By continuity of the mapping $(\xi, \eta) \mapsto \langle \nabla \phi^*(\xi) - \nabla \phi^*(0), \eta - 0 \rangle$ we know there exists $\delta >0$ such that for all $\xi, \eta \in B_{\|\cdot\|_*}(v^\star, \delta)$ one has,
$$
\langle \nabla \phi^*(\xi) - \nabla \phi^*(0), \eta - 0 \rangle > \varepsilon > 0.
$$
Since $\lim_{k \to \infty} v^k = v^\star$, 
for $k$ sufficiently large we have that for any $i,N \in \bN$, $v^{k+i},v^{k+N} \in B_{\|\cdot\|_*}(v^\star, \delta)$, where $B_{\|\cdot\|_*}(v^\star, \delta)$ denotes the open $\|\cdot\|_*$-ball with radius $\delta$ around $v^\star$, 
and hence
\begin{align} \label{eq:monotonicity_sum_lemma}
\langle \nabla \phi^*(v^{k+i}) - \nabla \phi^*(0), v^{k+N} - 0 \rangle > \varepsilon.
\end{align}
Combining \cref{eq:monotonicity_sum_lemma,eq:sum_lemma} yields that
\begin{align*}
\langle x^\star - x^{k+1+N}, v^{k+N} \rangle
&= \langle x^\star-x^k, v^{k+N} \rangle + \sum_{i=0}^{N} \langle \nabla \phi^*(v^{k+i}) - \nabla \phi^*(0), v^{k+N} - 0 \rangle \\
&\geq \langle x^\star-x^k, v^{k+N} \rangle + (N+1) \varepsilon,
\end{align*}
for any $k> K$, $K$ sufficiently large. Let now $k=K+1$ fixed and pass to the limit $\lim_{N \to \infty} v^{k+N} = v^\star$. Then we have that
$$
\lim_{N \to \infty} \langle x^\star-x^k, v^{k+N} \rangle = \langle x^\star-x^k, v^\star \rangle
$$
and $\lim_{N \to \infty}(N+1)\varepsilon = \infty$ and hence
$$
\lim_{N \to \infty} \langle x^\star - x^{k+1+N}, v^{k+N} \rangle = \infty.
$$
Since $v^{k+N} \in T(x^{k+1+N})$ and $x^\star \in T(0)$, by monotonicity of $T$,
we know that
$$
\langle x^\star - x^{k+1+N}, v^{k+N} \rangle \leq 0,
$$
a contradiction. Hence, $\lim_{k\to \infty} v^k = 0$.
\end{proof}

\begin{figure}[!t]
\centering
  \subfloat[]{
        \centering
        \resizebox{0.47\textwidth}{!}{
		\input{figures/plotcounterexample_euclidean}
 	    }
  }
  \subfloat[]{
        \centering
        \resizebox{0.47\textwidth}{!}{
		\input{figures/plotcounterexample_power}
 	    }
  }
\caption{Illustration of the proof of \cref{thm:standalone}. (a) illustrates the isotropic case and (b) displays the anisotropic case. In both subfigures the black circles represent the primal iterates $x^{k+1}= x^k - \nabla \phi^*(v^k)$ and the black arrows represent the (normalized) dual iterates $v^k$ that converge to $v$, and hence $\|v^k - v\|_* \leq \varepsilon$ for $k$ sufficiently large. In the visualization we chose $v^k=v$ constant for simplicity. The solid black lines represent the half-spaces $H_k$ with normals $v^k$, introduced in \cref{sec:projective_interpretation}, that separate the current iterates from the solution $x^\star$ indicated by the black cross. The dashed black curves represent the level-lines $\{ x \in X : \phi(x^k - x) = \phi(x^k - x^{k+1}) \}$ where $x^{k+1}$ is the anisotropic projection \cref{eq:anisotropic_halfspace_projection}. During the first 3 iterations $x^k$ and the solution $x^\star$ lie on opposite sides of the hyperplanes, as guaranteed by monotonicity of $T$. However, after 6 iterations the hyperplane $H_k$ does not separate the old iterate $x^k$ from the solution $x^\star$ and hence the algorithm ``overshoots'' the point $x^\star$. This contradicts the monotonicity of $T$ and hence the assumption that $v^k \to v \neq 0$ was wrong. Intuitively this means that the iterates cannot escape indefinitely in the same direction.}
\label{fig:visualization}
\end{figure}

In the previous standalone result we have shown that the primal iterates cannot escape indefinitely in the same direction. However, this does not exclude the possibility of the iterates being caught in limit cycles.
Moreover, the above result implicitly relies on the property that $\{v^k\}_{k =0}^\infty$ is a Cauchy sequence, which cannot be guaranteed a priori. Building upon the proof of the previous result we show in the next theorem that every limit point is a solution excluding limit cycles. For that purpose we exploit the fact that the Bregman distance $D_{\phi^*}(v^k, v^{k+1})$ between consecutive dual iterates $v^{k+1},v^k$ vanishes and assume that $\{x^k\}_{k=0}^\infty$ has a limit point.
The first part of the proof of the next result follows along the lines of \cref{thm:standalone}:
\begin{theorem}\label{thm:convergence_global_anisotropic}
Let $T: X \rightrightarrows X^*$ be maximal monotone. Then the following hold:
\begin{thmenum}
\item \label{thm:convergence_global_anisotropic:dual} Suppose that $\{x^k\}_{k=0}^\infty$ has a limit point, then $\lim_{k \to \infty} v^k = 0$.
\item \label{thm:convergence_global_anisotropic:limit_point} Every limit point of the sequence $\{x^k\}_{k=0}^\infty$ is a zero of $T$.
\end{thmenum}
\end{theorem}
\begin{proof}
``\labelcref{thm:convergence_global_anisotropic:dual}'': Following the first step of the proof of \cref{thm:standalone} we obtain the identity
\begin{align}\label{eq:identity_sum}
\langle x^\star - x^{k+1+N}, v^{k+N} \rangle
&= \langle x^\star-x^{k}, v^{k+N} \rangle + \sum_{i=0}^N \langle \nabla \phi^*(v^{k+i}) - \nabla \phi^*(0), v^{k+N} - 0 \rangle. 
\end{align}
Let $x^\infty$ be a limit point of $\{x^k\}_{k=0}^\infty$. Hence, there is a subsequence $x^{k_j} \to x^\infty$. In light of \cref{thm:dual_space_fejer:bounded} the dual sequence $\{v^k\}_{k=0}^\infty$ is bounded. Consequently, there exists a constant $\zeta > 0$ such that $\langle x^\star-x^{k_j}, v^{k_j+N} \rangle \geq -\zeta$ for all $j,N \in \bN$.
Then the identity in \cref{eq:identity_sum} yields
\begin{align}
\langle x^\star - x^{k_j+1+N}, v^{k_j+N} \rangle
&\geq -\zeta + \sum_{i=0}^N \langle \nabla \phi^*(v^{k_j+i}) - \nabla \phi^*(0), v^{k_j+N} - 0 \rangle, 
\end{align}
for all $j,N \in \bN$.
Suppose that $v^{k_j} \not\to 0$.
Since $\{v^{k}\}_{k=0}^\infty$ is bounded, by going to another subsequence if necessary, there is some $v^\star \in X^*$, $v^\star \neq 0$ such that $\lim_{j \to \infty} v^{k_j} = v^\star$.
By strict monotonicity of $\nabla \phi^*$ it holds that
$$
\langle \nabla \phi^*(v^{\star}) - \nabla \phi^*(0), v^{\star} - 0 \rangle = 2\varepsilon > 0,
$$
for some $\varepsilon > 0$. Choose $N \in \bN$ such that $N > \zeta/\varepsilon$ and keep it fixed.

By continuity of the mapping $(\xi, \eta) \mapsto \langle \nabla \phi^*(\xi) - \nabla \phi^*(0), \eta - 0 \rangle$ we know there exists $\delta >0$ such that for all $\xi, \eta \in B_{\|\cdot\|_*}(v^\star, \delta)$ we have that
\begin{align} \label{eq:continuity_eps}
\langle \nabla \phi^*(\xi) - \nabla \phi^*(0), \eta - 0 \rangle > \varepsilon > 0.
\end{align}
Thanks to \cref{thm:dual_space_fejer:sum}, $D_{\phi^*}(v^{k_j}, v^{k_j+1}) \to 0$. Since $\lim_{j \to \infty} v^{k_j} = v^\star$ in light of \cref{thm:props_bregman_dist:convergence} it holds that $\lim_{j \to \infty} v^{k_j+1} = v^\star$. Since $N> \zeta/\varepsilon$ is a fixed finite integer that does not depend on $k_j$, repeating the same argument $N$ times, we obtain that for any $0 \leq i \leq N$, $\lim_{j \to \infty} v^{k_j+i} = v^\star$. Consequently, for every $0 \leq i \leq N$ there exists $M_i >0$ such that for all $j\geq M_i$, $v^{k_j+i} \in B_{\|\cdot\|_*}(v^\star, \delta)$. Now choose $M(N)=\max_{1 \leq i \leq N} M_i$. Then, for any $j \geq M(N)$ and every $1 \leq i \leq N$, $v^{k_j+i} \in B_{\|\cdot\|_*}(v^\star, \delta)$ and hence we obtain using \cref{eq:continuity_eps}
$$
\langle \nabla \phi^*(v^{{k_j}+i}) - \nabla \phi^*(0), v^{{k_j}+N} - 0 \rangle > \varepsilon.
$$
This means that for $j\geq M(N)$, by the choice of $N > \zeta/\varepsilon$, one obtains
\begin{align}
\langle x^\star - x^{k_j+1+N}, v^{k_j+N} \rangle
&\geq -\zeta + \sum_{i=0}^N \langle \nabla \phi^*(v^{k_j+i}) - \nabla \phi^*(0), v^{k_j+N} - 0 \rangle \\
&\geq -\zeta +(N+1) \varepsilon > -\zeta + \tfrac{\zeta}{\varepsilon}\varepsilon + \varepsilon= \varepsilon.
\end{align}
Since $v^{k_j + N} \in T( x^{k_j+1+N}),$ we know by monotonicity of $T$ that
$$
\langle x^\star - x^{k_j+1+N}, v^{k_j+N} \rangle \leq0,
$$
a contradiction. Thus $\lim_{j \to \infty} v^{k_j} = 0$ and by \cref{thm:props_bregman_dist:consitency}, $\lim_{j \to \infty} D_{\phi^*}(v^{k_j}, 0) = 0$. In light of \cref{thm:dual_space_fejer:convergence}, $\{D_{\phi^*}(v^k, 0)\}_{k=0}^\infty$ converges. Thus,
$$
\lim_{k \to \infty} D_{\phi^*}(v^k, 0) = \lim_{j \to \infty} D_{\phi^*}(v^{k_j}, 0) = 0.
$$
Invoking \cref{thm:props_bregman_dist:convergence} we obtain $\lim_{k \to \infty} v^k =0$.

``\labelcref{thm:convergence_global_anisotropic:limit_point}'': Let $x^\infty$ be a limit point of the sequence $\{x^k\}_{k=0}^\infty$. Hence there exists a subsequence indexed by $k_j$ along which $\lim_{j \to \infty} x^{k_j} = x^\infty$. In light of \labelcref{thm:convergence_global_anisotropic:dual}, $\lim_{k \to \infty} v^k \to 0$. Since $v^{k_j-1} \in T(x^{k_j})$, by maximal monotonicity of $T$, we have that $0 \in T(x^\infty)$.
\end{proof}

Boundedness of the primal iterates $x^k$, and hence existence of a limit point, can be guaranteed if $T$ is coercive as shown in the following corollary:
\begin{corollary} \label{thm:coercive}
Assume that $T : X \rightrightarrows X^*$ is maximal monotone and coercive. Then the sequence $\{x^k\}_{k=0}^\infty$ is bounded. Therefore $\lim_{k \to \infty} v^k = 0$ and $\{x^k\}_{k=0}^\infty$ has all its limit points in $\zer T$.
\end{corollary}
\begin{proof}
In light of \cref{thm:dual_space_fejer:bounded} the dual sequence $\{v^k\}_{k=0}^\infty$ is bounded. Suppose that the primal sequence $\{x^k\}_{k=0}^\infty$ is unbounded. This implies the existence of a subsequence indexed by $k_j$ along which $\lim_{j \to \infty} \|x^{k_j}\| = \infty$. Then, by coercivity of $T$, $\infty =\lim_{j \to \infty} \inf \{\|v\|_* : v \in T (x^{k_j}) \} \leq \lim_{j \to \infty} \|v^{k_j-1}\|_*$, a contradiction. Hence, the primal sequence $\{x^k\}_{k=0}^\infty$ is bounded as well and therefore it has a limit point. Invoking \cref{thm:convergence_global_anisotropic} we obtain the claimed result.
\end{proof}
\begin{remark}[coercivity and uniform monotonicity]
    Let $T$ be uniformly monotone with a super-coercive modulus $\omega$. Then $T$ is coercive by \cite[Proposition 22.11]{BaCo110}. In this case the above theorem implies the conclusion of \cref{thm:uniform_monotonicity:convergence}.
    In light of \cite[Proposition 4.1]{liu2023strongly} the modulus of $T$ is super-coercive if $\dom T$ is convex. Note that this does not hold in general even if $T$ is maximal monotone.
\end{remark}
To conclude, we provide an example for a coercive operator in the context of minimax problems:
\begin{example} \label{ex:coercive_primal_dual}
Consider the optimization problem:
\begin{align}
\min_{x \in \bR^n} f(x) + g(Ax),
\end{align}
for some $A \in \bR^{m \times n}$, $f \in \Gamma_0(\bR^n)$ with $\dom f$ bounded and $g \in \Gamma_0(\bR^m)$ with $\dom g = \bR^m$.
This is equivalent to the following saddle-point problem
$$
\min_{x \in \bR^n} \max_{y \in \bR^m} \big\{L(x,y):=f(x) + \langle A x, y\rangle -g^*(y)\big\}.
$$
Since $g$ has full domain, $g^*$ is super-coercive by \cite[Theorem 11.8(d)]{RoWe98}. Hence, the saddle-point operator $T(x,y):=\partial_x L(x,y) \times \partial_y(-L)(x,y)$ with
$$
T^{-1}(0) = \argminimax_{x \in \bR^n,y \in \bR^m} L(x,y)
$$
is coercive. Thus the assumptions in \cref{thm:coercive} are satisfied. Note that from the perspective of the PPA iteration, both $x$ and $y$ are considered the primal variables.
\end{example}

\subsubsection{Local convergence under H\"older growth}
Finally, we analyze the local convergence rate of the method specializing $\phi=\frac{1}{p}\|\cdot\|_p^p$ for $p>1$. Denote
by
\begin{equation}
\adist{\zer T}{x}:=\inf_{y \in \zer T} \|x-y\|_p,
\end{equation}
and assume that the following growth property holds true:
    \begin{align} \label{eq:growth_anisotropic}
\exists \delta,\nu, \rho > 0 : x \in T^{-1}(v), \|v\|_q < \delta \Rightarrow \adist{\zer T}{x} \leq \rho \|v\|_q^\nu.
\end{align}
Note that the above property is more restrictive than \cref{eq:growth_euclidean} as we require the bound to hold for any $x$ and not only for points that are near a solution $x^\star \in \zer T$.
\begin{example} \label{ex:growth_anisotropic}
Let $p\geq2$ and $\frac{1}{p}+\frac{1}{q}=1$.
Choose $T$ as in \cref{ex:growth_euclidean}.
Then $T^{-1}(0)$
is single-valued and hence $\adist{\zer T}{x}=\|x - x^\star\|_p$. We have that $\|Ax\|_q=\alpha\|x\|_q$ and $Ax^\star = b$.
By equivalence of the $\ell^p$-norm and the $\ell^q$-norm, since $p\geq q$, we have that
$$
\|x-x^\star\|_p \leq \|x - x^\star \|_q = \tfrac{1}{\alpha}\|Ax - b\|_q = \tfrac{1}{\alpha}\|T(x)\|_q,
$$
and hence \cref{eq:growth_anisotropic} holds for $\delta=\infty$, $\rho=\tfrac{1}{\alpha}$ and $\nu =1$.
\end{example}
\begin{proposition} \label{thm:convergence_rate_anisotropic}
Let $\phi(x)=\frac{1}{p}\|x\|_p^p$ for $p >1$. Assume that \cref{eq:growth_anisotropic} holds and that $\{x^k\}_{k=0}^\infty$ has a limit point. Choose $p > \frac{1}{\nu} +1$. Then the algorithm converges  with order $\nu(p-1)$ and rate $\rho$, i.e., for $k\geq K$, $K$ sufficiently large we have $\adist{\zer T}{x^{k+1}} \leq \rho \adist[\nu(p-1)]{\zer T}{x^{k}}$.
\end{proposition}
\begin{proof}
Let $x \in X$. 
By convex duality we have that $\phi^*(\nabla \phi(x))= \langle x, \nabla \phi(x) \rangle - \phi(x)$.
Since $\phi(x)=\tfrac{1}{p}\|x\|_p^p$, $\langle x, \nabla \phi(x) \rangle = \|x\|_p^p$ and $\phi^*(v)=\frac{1}{q}\|v\|_q^q$ for $q=\frac{p}{p-1}$ we have that $\phi^*(\nabla \phi(x)) = \frac{1}{q}\|\nabla \phi(x)\|_q^q= (1-\tfrac{1}{p})\|x\|_p^p$, and thus
\begin{align} \label{eq:identity_pnorm_q}
\|\nabla \phi(x)\|_q = \|x\|_p^{p-1}.
\end{align}
In light of \cref{thm:convergence_global_anisotropic:dual}, $T(x^{k+1})\ni v^k \to 0$. Hence for $k \geq K$, $K$ sufficiently large we have that
$
v^k = \nabla \phi(x^k -x^{k+1}) \in B_{\|\cdot\|_q}(0, \delta),
$
and thus using \cref{eq:growth_anisotropic,eq:identity_pnorm_q},
\begin{align} \label{eq:bound_growth_anisotropic}
\adist{\zer T}{x^{k+1}} \leq \rho\|v^k\|_q^\nu =\rho \|x^k -x^{k+1}\|_p^{\nu(p-1)}.
\end{align}
Since $\zer T \subseteq H_k$ and in light of \cref{thm:halfspace:projection},
$$
x^{k+1}=\aproj{H_k}(x^k)=\argmin_{x \in H_k} \|x^k - x\|_p,
$$
we have that
\begin{align} \label{eq:bound_anisotropic_2}
\|x^k - x^{k+1}\|_p=\adist{H_k}{x^k}= \inf_{x \in H_k} \|x^k - x\|_p \leq \inf_{x \in \zer T}\|x^k - x\|_p = \adist{\zer T}{x^{k}}.
\end{align}
Combining \cref{eq:bound_anisotropic_2,eq:bound_growth_anisotropic} we obtain:
\[
\adist{\zer T}{x^{k+1}} \leq \rho \|x^k -x^{k+1}\|_p^{\nu(p-1)} \leq \rho \adist[\nu(p-1)]{\zer T}{x^{k}}. \qedhere
\]
\end{proof}
The convergence of anisotropic PPA with higher order using the setup in \cref{ex:growth_anisotropic} is depicted in \cref{fig:toilette} for $p=3$.
It should be noted that an error tolerance of the form \cref{eq:error_bound_isotropic} can be incorporated here as well by using the equivalence of the $\ell^p$-norm and the $\ell^2$-norm in a finite-dimensional space. However, since we would like to avoid the usage of $\ell^2$-norms in the anisotropic case, this is omitted here.

\section{Application: Anisotropic proximal augmented Lagrangian} \label{sec:alm}
In this section we provide an example for using the anisotropic proximal point scheme in the context of a generalization of the proximal augmented Lagrangian method \cite[Section 5]{rockafellar1976augmented}.

We revisit the setup in \cref{ex:coercive_primal_dual} where $f,g$ are general convex functions: We consider the minimization problem
\begin{align} \label{eq:minimization_problem}
\inf_{x \in \bR^n} f(x)+g(Ax),
\end{align}
for $f\in \Gamma_0(\bR^n)$, $g \in \Gamma_0(\bR^m)$ and $A \in \bR^{m \times n}$.
The associated saddle-point problem is given as
\begin{align}
\inf_{x \in \bR^n} \sup_{y \in \bR^m} ~L(x,y),
\end{align}
for the convex-concave Lagrangian $L(x,y)=f(x) + \langle Ax, y\rangle - g^*(y)$. Consider the monotone mapping
\begin{align}
T(x,y) = \partial_x L(x,y) \times \partial_y (-L)(x,y) = (\partial f(x) + A^\top y )\times (\partial g^*(y) - Ax),
\end{align}
and choose the separable prox-function $\phi(x,y):=\varphi(x) + \vartheta(y)$ for Legendre functions $\varphi, \vartheta$. We apply \cref{eq:update} with under-relaxation parameter $\lambda =1$ and prox-function $\phi$, i.e., $(x^{k+1}, y^{k+1}) = (\id + \nabla \phi^* \circ T)^{-1}(x^k, y^k)$. Thanks to the separability of $\phi$ in $x$ and $y$, the update can be rewritten in terms of the system of inclusions
\begin{subequations}
\begin{align}
  \nabla \vartheta(x^k-x^{k+1}) &\in \partial_x L(x^{k+1},y^{k+1}) \\
  \nabla \varphi(y^k - y^{k+1}) &\in \partial_y (-L)(x^{k+1},y^{k+1}). \label{eq:multiplier_inclusion}
\end{align}
\end{subequations}

\begin{table}[h]
\centering

\begin{tabular}{lccccccc}
\toprule
Dataset & $n$ & $m$ & maxit total & maxit step & $\Delta_{\mathrm{rel}}$ & $r_{\mathrm{rel}}$ & $\varepsilon$ \\
\midrule
CONT-050   & 2597  & 2401 & 15000 & 1000 & $10^{-4}$ & $10^{-4}$ & $10^{-8}$ \\
CONT-100   & 10197 & 9801 & 50000 & 2000 & $10^{-5}$ & $10^{-4}$ & $10^{-6}$ \\
CVXQP1\_M  & 1000  & 500  & 50000 & 800  & $10^{-6}$ & $10^{-6}$ & $10^{-8}$ \\
CVXQP2\_S  & 100   & 25   & 3000  & 100  & $10^{-6}$ & $10^{-6}$ & $10^{-8}$ \\
GOULDQP2   & 699   & 349  & 2000  & 8    & $10^{-4}$ & $10^{-5}$ & $10^{-5}$ \\
MOSARQP1   & 3200  & 700  & 5000  & 120  & $10^{-6}$ & $10^{-4}$ & $10^{-4}$ \\
MOSARQP2   & 1500  & 600  & 12000 & 250  & $10^{-6}$ & $10^{-6}$ & $10^{-8}$ \\
\bottomrule
\end{tabular}

\caption{\label{tab:maros_data}Specifics of a subset of the Maros--M{\'e}sz{\'a}ros dataset for convex quadratic programs \cite{maros1999repository} and shared optimizer hyperparameters. $n$ is the numbers of decision variables and $m$ the number of constraints. maxit total is the maximum number of cumulative L-BFGS-B steps and maxit step the maximum number of inner iterations per ALM step. We denote by $\Delta_{\mathrm{rel}}=|f-f_\star|/(1+|f_\star|)$ the desired relative suboptimality and $r_{\mathrm{rel}}=\|Ax^\star -b\|_\infty / (1+\|b\|_\infty)$ denotes the desired relative constraint violation. $\varepsilon$ is a configurable parameter for the termination criterion of the inner solver.}
\end{table}

\begin{table}[h]
\centering

\begin{tabular}{lccccc}
\toprule
Dataset & $p$ & $\tau$ & $\sigma$ & req. iters total & req. iters outer \\
\midrule
\multirow{4}{*}{CONT-050}
& 3 & $10^{3}$ & $1$        & $\mathbf{8540}$  & $\mathbf{9}$ \\
& 2 & $10^{3}$ & $1$        & --               & -- \\
& 2 & $10^{5}$ & $10$       & $9290$           & $12$ \\
& 2 & $10^{5}$ & $10^{2}$   & $14900$          & $15$ \\
\midrule
\multirow{4}{*}{CONT-100}
& 3 & $10^{3}$ & $10$       & $\mathbf{49200}$ & $\mathbf{25}$ \\
& 2 & $10^{3}$ & $10$       & --               & -- \\
& 2 & $10^{5}$ & $10^{2}$   & --               & -- \\
& 2 & $10^{5}$ & $10^{3}$   & --               & -- \\
\midrule
\multirow{4}{*}{CVXQP1\_M}
& 3 & $10^{2}$ & $10^{5}$   & $\mathbf{20200}$ & $\mathbf{30}$ \\
& 2 & $10^{2}$ & $10^{5}$   & --               & -- \\
& 2 & $10^{5}$ & $10^{8}$   & --               & -- \\
& 2 & $10^{5}$ & $10^{10}$  & --               & -- \\
\midrule
\multirow{4}{*}{CVXQP2\_S}
& 3 & $10^{2}$ & $700$      & $\mathbf{239}$   & $\mathbf{5}$ \\
& 2 & $10^{2}$ & $700$      & $370$            & $21$ \\
& 2 & $10^{5}$ & $1000$     & $344$            & $22$ \\
& 2 & $10^{5}$ & $2000$     & $324$            & $22$ \\
\midrule
\multirow{4}{*}{GOULDQP2}
& 3 & $10^{5}$ & $10^{-1}$  & $\mathbf{30}$    & $\mathbf{10}$ \\
& 2 & $10^{5}$ & $10^{-1}$  & --               & -- \\
& 2 & $10^{5}$ & $1$        & $73$             & $29$ \\
& 2 & $10^{5}$ & $10$       & $86$             & $12$ \\
\midrule
\multirow{4}{*}{MOSARQP1}
& 3 & $10^{3}$ & $1$        & $2780$           & $\mathbf{147}$ \\
& 2 & $10^{3}$ & $1$        & $\mathbf{1610}$  & $314$ \\
& 2 & $10^{5}$ & $10$       & $3520$           & $249$ \\
& 2 & $10^{5}$ & $10^{2}$   & $4240$           & $205$ \\
\midrule
\multirow{4}{*}{MOSARQP2}
& 3 & $10^{3}$ & $10$       & $\mathbf{819}$   & $\mathbf{13}$ \\
& 2 & $10^{3}$ & $10$       & $2170$           & $174$ \\
& 2 & $10^{5}$ & $10^{2}$   & $3950$           & $141$ \\
& 2 & $10^{5}$ & $10^{3}$   & $9770$           & $56$ \\
\bottomrule
\end{tabular}

\caption{\label{tab:maros_performance}Performance comparison of the anisotropic proximal augmented Lagrangian
method ($p=3$) and the classical proximal augmented Lagrangian method \cite{rockafellar1976augmented} ($p=2$) using L-BFGS-B as an inner solver on a subset of the Maros--M{\'e}sz{\'a}ros dataset for convex quadratic programs \cite{maros1999repository}. Bold values indicate best performing and ``--'' means that the desired accuracy has not been achieved within the maximum number of total L-BFGS-B steps.}
\end{table}

\begin{figure}[!t]
\centering
\subfloat[rel. suboptimality $\Delta_{\mathrm{rel}}$]{
        \centering
        \resizebox{0.47\textwidth}{!}{
\begin{tikzpicture}

\definecolor{darkgray176}{RGB}{176,176,176}
\definecolor{lightgray204}{RGB}{204,204,204}

\begin{axis}[
legend cell align={left},
legend style={fill opacity=0.8, draw opacity=1, text opacity=1, draw=lightgray204},
legend style={font=\scriptsize, legend cell align=left, align=left, draw=white!15!black},
log basis y={10},
tick align=outside,
tick pos=left,
x grid style={darkgray176},
xmajorgrids,
xmin=-450, xmax=9450,
xtick style={color=black},
xticklabel style={font=\scriptsize},
y grid style={darkgray176},
ymajorgrids,
ymin=6.2233980456122e-07, ymax=1.24549808469113,
ymode=log,
ytick style={color=black},
ytick={1e-08,1e-07,1e-06,1e-05,0.0001,0.001,0.01,0.1,1,10,100},
yticklabel style={font=\scriptsize},
yticklabels={
  \(\displaystyle {10^{-8}}\),
  \(\displaystyle {10^{-7}}\),
  \(\displaystyle {10^{-6}}\),
  \(\displaystyle {10^{-5}}\),
  \(\displaystyle {10^{-4}}\),
  \(\displaystyle {10^{-3}}\),
  \(\displaystyle {10^{-2}}\),
  \(\displaystyle {10^{-1}}\),
  \(\displaystyle {10^{0}}\),
  \(\displaystyle {10^{1}}\),
  \(\displaystyle {10^{2}}\)
}
]
\addplot [line width=1pt, black, mark=triangle*, mark repeat=1, mark options={solid}]
table {%
0 0.644050071663351
1000 0.0226753446910056
2000 0.00406440407810321
3000 0.00563498720283183
4000 0.00395557206954081
5000 0.00257880767442039
6000 0.00140876873628672
6811 0.000584501368712005
7605 0.000137831058621014
8537 1.36489164758361e-06
8565 1.35904805598609e-06
8743 1.26993744449887e-06
8835 1.21325687615458e-06
9000 1.20351362217255e-06
};
\addlegendentry{$p=3,\sigma=1,\tau=10^{3}$}
\addplot [line width=1pt, black, dash pattern=on 1pt off 3pt on 3pt off 3pt, mark=x, mark repeat=1, mark options={solid}]
table {%
0 0.644050071663351
756 0.311170269278126
1337 0.156664347700687
1899 0.0781183538564861
2397 0.0383963923876759
2885 0.0182904686981014
3346 0.00796575651683651
3917 0.00217591005278045
4402 0.00127394448778916
4889 0.00329176747561151
5272 0.00444164753229599
5629 0.00508062254613727
5980 0.00540982832836475
6305 0.00554149360273715
6623 0.00554106128773822
6984 0.00544888744138906
7268 0.00529222790790271
7520 0.0050858873618656
7873 0.00485022391256534
8220 0.00459191221204406
8615 0.00431924774637866
8968 0.00403597975673273
9000 0.00375824026780079
};
\addlegendentry{$p=2,\sigma=1,\tau=10^{3}$}
\addplot [line width=1pt, black, dashed, mark=pentagon*, mark repeat=1, mark options={solid}]
table {%
0 0.644050071663351
1000 0.0355930474840963
2000 0.00300563218126928
3000 0.00360823137665628
4000 0.00177341162309248
5000 0.00166525822459831
5979 0.00103116185371963
6970 0.000579038817237334
7761 0.000328664558781289
8498 0.000174636622960326
8528 0.000174490718178409
8557 0.000174381357324135
9000 1.37599318126732e-05
};
\addlegendentry{$p=2,\sigma=10,\tau=10^{5}$}
\addplot [line width=1pt, black, dash pattern=on 1pt off 3pt on 3pt off 3pt, mark=diamond*, mark repeat=1, mark options={solid}]
table {%
0 0.644050071663351
1000 0.400985011226858
1901 0.0134024695918028
2901 0.00866693354334504
3901 0.00901849529596102
4901 0.00115790826423232
5901 0.00165197416491337
6901 0.000251155866830538
7901 0.000578545966757004
8901 9.69831606900095e-05
9000 9.1121317590124e-05
};
\addlegendentry{$p=2,\sigma=10^{2},\tau=10^{5}$}
\end{axis}

\end{tikzpicture}
 	   }
  }
\subfloat[rel. constraint violation $r_{\mathrm{rel}}$]{
        \centering
        \resizebox{0.47\textwidth}{!}{
\begin{tikzpicture}

\definecolor{darkgray176}{RGB}{176,176,176}
\definecolor{lightgray204}{RGB}{204,204,204}

\begin{axis}[
legend cell align={left},
legend style={fill opacity=0.8, draw opacity=1, text opacity=1, draw=lightgray204},
legend style={font=\scriptsize, legend cell align=left, align=left, draw=white!15!black},
log basis y={10},
tick align=outside,
tick pos=left,
x grid style={darkgray176},
xmajorgrids,
xmin=-450, xmax=9450,
xtick style={color=black},
xticklabel style={font=\scriptsize},
y grid style={darkgray176},
ymajorgrids,
ymin=1.1296718537395e-06, ymax=25.6484467018515,
ymode=log,
ytick style={color=black},
ytick={1e-07,1e-06,1e-05,0.0001,0.001,0.01,0.1,1,10,100,1000},
yticklabel style={font=\scriptsize},
yticklabels={
  \(\displaystyle {10^{-7}}\),
  \(\displaystyle {10^{-6}}\),
  \(\displaystyle {10^{-5}}\),
  \(\displaystyle {10^{-4}}\),
  \(\displaystyle {10^{-3}}\),
  \(\displaystyle {10^{-2}}\),
  \(\displaystyle {10^{-1}}\),
  \(\displaystyle {10^{0}}\),
  \(\displaystyle {10^{1}}\),
  \(\displaystyle {10^{2}}\),
  \(\displaystyle {10^{3}}\)
}
]
\addplot [line width=1pt, black, mark=triangle*, mark repeat=1, mark options={solid}]
table {%
0 11.8766098186697
1000 0.00405402887776463
2000 0.00535787708597564
3000 0.00787662047122348
4000 0.00615196124233494
5000 0.00390478812357732
6000 0.00153081196811226
6811 0.000592197645215326
7605 0.000134086255235276
8537 2.43961271554719e-06
8565 2.48022410854348e-06
8743 4.03708430926586e-06
8835 4.6658834838115e-06
9000 2.68205300429131e-06
};
\addlegendentry{$p=3,\sigma=1,\tau=10^{3}$}
\addplot [line width=1pt, black, dash pattern=on 1pt off 3pt on 3pt off 3pt, mark=x, mark repeat=1, mark options={solid}]
table {%
0 11.8766098186697
756 0.0101653982737134
1337 0.00768554031579982
1899 0.00371662836716623
2397 0.000893045530489038
2885 0.00103072234302994
3346 0.00167459083732572
3917 0.00793650793650794
4402 0.00793650793650794
4889 0.00793650793650794
5272 0.00793650793650794
5629 0.00793650793650794
5980 0.00793650793650794
6305 0.00793650793650794
6623 0.00793650793650794
6984 0.00793650793650794
7268 0.00793650793650794
7520 0.00793650793650794
7873 0.00721139516311511
8220 0.00651912189860785
8615 0.00529175698504417
8968 0.00475172732181073
9000 0.00436813622916966
};
\addlegendentry{$p=2,\sigma=1,\tau=10^{3}$}
\addplot [line width=1pt, black, dashed, mark=pentagon*, mark repeat=1, mark options={solid}]
table {%
0 11.8766098186697
1000 0.0162399708797693
2000 0.0209430101790863
3000 0.00533924625140874
4000 0.00345589929146943
5000 0.00377629493007538
5979 0.0030412129346522
6970 0.000588169614117504
7761 0.000333313158191686
8498 0.000175346030801421
8528 0.000176832297754406
8557 0.000180420251013419
9000 5.30762706871555e-05
};
\addlegendentry{$p=2,\sigma=10,\tau=10^{5}$}
\addplot [line width=1pt, black, dash pattern=on 1pt off 3pt on 3pt off 3pt, mark=diamond*, mark repeat=1, mark options={solid}]
table {%
0 11.8766098186697
1000 0.00156843426671173
1901 0.00677107394679379
2901 0.00656775774756452
3901 0.000360359294224838
4901 0.00107429743998616
5901 0.000896600550952231
6901 0.000306848712517554
7901 0.000338848755374705
8901 0.00010030077334958
9000 0.000149574781320416
};
\addlegendentry{$p=2,\sigma=10^{2},\tau=10^{5}$}
\end{axis}

\end{tikzpicture}
 	   }
  }
\caption{
 Performance comparison of the anisotropic proximal augmented Lagrangian method ($p=3$) and the classical proximal augmented Lagrangian method \cite{rockafellar1976augmented} ($p=2$) using L-BFGS-B as an inner solver on CONT-050, a representative instance of the Maros--M{\'e}sz{\'a}ros dataset for convex quadratic programs \cite{maros1999repository}; see first row in \cref{tab:maros_performance}.
 We plot the cumulative number of inner iterations on the horizontal axis. Since the number of inner steps dominates the computational cost it serves as a proxy for the runtime of the method. Each mark signifies one outer iteration.
 }
\label{fig:maros}
\end{figure}

which is equivalent to the following minimax problem
\begin{align} \label{eq:prox_minimax}
(x^{k+1},y^{k+1}) = \argminimax_{x \in\bR^n, y \in \bR^m} ~L(x,y) + \vartheta(x^k - x) - \varphi(y^k- y).
\end{align}
We introduce the $\varphi$-augmented Lagrangian $L_{\varphi}(x,y)$ function:
\begin{align} \label{eq:augmented_lagrangian}
    L_{\varphi}(x,y) &:= \sup_{\eta \in \bR^m} L(x, \eta) - \varphi(y-\eta).
\end{align}
The $\varphi$-augmented Lagrangian can be written in terms of Rockafellar's perturbation framework for convex duality in terms of the dualizing parametrization
\begin{equation}
    F(x,u) = f(x) + g(Ax + u).
\end{equation}
Then the Lagrangian can be written as the negative partial convex conjugate of $F$
\begin{equation}
L(x,y)=\inf_{u \in \bR^m} F(x, u)-\langle u,y \rangle.
\end{equation}
An interchange of $\inf$ and $\sup$ in \cref{eq:augmented_lagrangian} leads to
\begin{align}
    L_{\varphi}(x,y) &= \sup_{\eta \in \bR^m} L(x, \eta) - \varphi(y-\eta) \notag\\
    &=\inf_{u \in \bR^m} F(x, u) + \sup_{\eta \in \bR^m}\langle u,-\eta \rangle - \varphi(y-\eta) \notag\\
    &=\inf_{u \in \bR^m} F(x, u) + \sup_{\xi \in \bR^m}\langle u,\xi \rangle - \varphi(\xi + y) \notag\\
    &=\inf_{u \in \bR^m} F(x, u) - \langle u, y\rangle + \varphi^*(u),
\end{align}
which is a generalized augmented Lagrangian in the sense of \cite[Definition 11.55]{RoWe98}. Note that for $\varphi^*=\frac{1}{q}\|\cdot\|^q$, and $q \to 1$ this approaches a sharp Lagrangian \cite[Example 11.58]{RoWe98}.

Using the $\varphi$-augmented Lagrangian function the update \cref{eq:prox_minimax} separates into primal and dual updates:
\begin{subequations}
\begin{align}
x^{k+1} &= \argmin_{x \in \bR^n} \big\{L_{\varphi}(x,y) + \vartheta(x^k -x)
\big\} \label{eq:primal_update}\\
y^{k+1} &=y^k - \nabla \varphi^*(-\nabla_y L_\varphi(x^{k+1}, y^k)). \label{eq:dual_mulitplier_update}
\end{align}
\end{subequations}
By \cref{eq:augmented_lagrangian}, $L_\varphi(x^{k+1},y^{k})$ is the negative infimal convolution of $(-L)(x^{k+1},\cdot)$ and $\varphi$ at $y^{k}$ which is exact for $y^{k+1}$ by \cref{eq:multiplier_inclusion}. Invoking \cite[Proposition 18.7]{BaCo110} we have that $L_\varphi(x^{k+1}, \cdot)$ is smooth and $\nabla \varphi(y^k-y^{k+1}) = \nabla_y (-L)_\varphi(x^{k+1},y^{k})$ which yields the dual update in \cref{eq:dual_mulitplier_update}.

For $\varphi=\vartheta=\frac{1}{2}\|\cdot\|_2^2$ this is the classical proximal augmented Lagrangian method \cite{rockafellar1976augmented}, in which not only the dual update but also the primal update is regularized with a quadratic prox-function.
Next we assess the performance of the anisotropic proximal augmented Lagrangian method on solving convex quadratic programs in standard form. In the problem formulation \cref{eq:minimization_problem} this corresponds to choices $f(x)=\frac{1}{2}\langle x, Qx\rangle + \langle x, c\rangle + \delta_\mathcal{B}(x)$ for a symmetric positive definite matrix $Q\in \bR^{n \times n}$, vector $c \in \bR^n$ and box constraints $\mathcal{B}=\{x \in \bR^n : l_i \leq x_i \leq u_i\}$ and $g=\delta_{\{b\}}$, $b\in \bR^m$.
In \cref{tab:maros_performance} we compare the anisotropic proximal ALM as described above and the classical proximal ALM \cite{rockafellar1976augmented} on a subset of the Maros--M{\'e}sz{\'a}ros dataset for convex quadratic programs \cite{maros1999repository} in terms of constraint violation and suboptimality. Both methods are initialized with the same random initial guesses for $x$ and $y$. In \cref{fig:maros} we plot the suboptimality gap and constraint violation for one representative instance. The ground-truth optimal values are obtained using QPALM \cite{hermans2022qpalm} in a highest-accuracy configuration.
We choose the primal regularization $\vartheta(x)=\tau \star h(x)$ and the dual regularization $\varphi(y)=\sigma \star h(y)$ for $h(x)=\frac{1}{p}\|x\|_2^p$. The parameters $\tau,\sigma$ correspond to primal and dual step-sizes respectively. We choose reasonable step-sizes for the anistropic proximal ALM $p=3$ by trial and error and consider a variety of step-size combinations for the baseline $p=2$. We convert the QPs into standard form and solve the resulting subproblem \cref{eq:primal_update} with L-BFGS-B with memory 25 using the implementation of \texttt{SciPy}. Note that the subproblem \cref{eq:primal_update} is locally H\"older-smooth and therefore also amenable to recent first-order solvers for H\"older-smooth optimization problems \cite{oikonomidis2024adaptive}. We terminate the inner solver when the normalized difference of function values $(L_k(x^{t}) - L_k(x^{t+1}))/\max\{|L_k(x^{t})|,|L_k(x^{t+1})|,1\}$ for $L_k(x)=L_\varphi(x,y^k)+\vartheta(x^k -x)$ is smaller than $\texttt{ftol} =\frac{\varepsilon}{(k+1)^p}$ for a predefined value of $\varepsilon$ where $k$ is the outer iteration counter and $t$ the inner iteration counter.
Since the cumulative number of L-BFGS-B steps dominates the computational cost of the method we use it as a proxy for the runtime. The dimensions of the individual problem instances and the shared parameters for the termination criterion of the inner solver are provided in \cref{tab:maros_data}.
The anisotropic proximal ALM with cubic regularization often requires fewer L-BFGS-B steps to converge than its Euclidean counterpart. Code for reproducing the experiments is available online\footnote{\url{https://github.com/EmanuelLaude/anisotropic-ppa}}, and further evidence appears in \cite{oikonomidis2025global}, which studies the anisotropic ALM without primal regularization. The non-quadratic proximal terms act as implicit primal and dual step-size selection in the classical proximal ALM which also aligns with \cref{rem:implicite_step_size}. This reduces the parameter dependence of adaptive proximal ALM with a handcrafted schedule for the step-sizes. While not intended as a state-of-the-art QP solver, which would require a more accurate subproblem solver like Newton's method in QPALM \cite{hermans2022qpalm}, the proposed scheme is better suited for large-scale problems, using only first-order information and avoiding large matrix inversions. A limit case of the proposed method has also been applied to plasma optimization in quasi-isodynamic stellarators for magnetic confinement fusion \cite{cadena2025constellaration,navarro2025optimizing}: The trust-region-like box-constraints used in the primal update arise as a limit of our framework $\vartheta=\tau \star(\frac{1}{p}\|\cdot\|_p^p)$ for $p \to \infty$.

\section{Conclusion} \label{sec:conclusion}
In this paper we have considered dual space nonlinear preconditioning for the proximal point method beyond the isotropic case studied by Luque. This is based on an anisotropic generalization of the resolvent of a monotone mapping which enjoys a duality relation with the well-studied Bregman resolvent. Connections between Bregman resolvents, the Bregman--Yosida regularization and convolutions of monotone operators also known as parallel sums were established. Using a Fej\'er monotonicity property in the dual space, subsequential convergence is established. Local convergence rates are investigated showing that under a local H\"older growth and a suitable preconditioner convergence rates of arbitrary order are attained.
An anisotropic generalization of the proximal augmented Lagrangian method is  described and applied to quadratic programming.

\printbibliography
\end{document}